\documentclass[12pt,reqno]{amsart}
\usepackage[foot]{amsaddr}
\usepackage[
  hmarginratio={1:1},     
  vmarginratio={1:1},     
  textwidth=430pt,        
  heightrounded,          
]{geometry}
\usepackage[utf8]{inputenc}
\usepackage[T1]{fontenc}
\usepackage{microtype}
\usepackage[shortlabels]{enumitem}

\usepackage{comment}
\usepackage[backgroundcolor=lightgray,obeyFinal]{todonotes}

\usepackage{amsmath,amssymb,amsthm}
\usepackage{math-macros}

\usepackage[maxbibnames=99,style=numeric,sorting=nyt]{biblatex}
\renewbibmacro{in:}{}

\usepackage{booktabs}
\usepackage{caption}
\numberwithin{figure}{section}
\numberwithin{table}{section}

\usepackage[bookmarksnumbered,hidelinks]{hyperref}
\usepackage[capitalize,noabbrev]{cleveref}

\theoremstyle{plain}
\newtheorem{theorem}{Theorem}[section]
\newtheorem{proposition}[theorem]{Proposition}
\newtheorem{lemma}[theorem]{Lemma}
\theoremstyle{definition}
\newtheorem{definition}[theorem]{Definition}
\newtheorem{problem}[theorem]{Problem}
\theoremstyle{remark}
\newtheorem{example}[theorem]{Example}
\newtheorem{remark}[theorem]{Remark}

\usepackage{tikz}
\usetikzlibrary{arrows.meta,calc,positioning}
\usetikzlibrary{cd}
\tikzcdset{arrow style=tikz}

\usetikzlibrary{arrows,decorations.markings}
\makeatletter
\tikzset{nomorepostaction/.code=\let\tikz@postactions\pgfutil@empty}
\makeatother
\tikzset{vertex/.style={shape=circle,fill,inner sep=0,minimum size=0.333em}}
\tikzset{edge/.style={draw,decoration={
    markings,
    mark=at position 0.6 with {\arrow[scale=1.2]{latex'}}},
  postaction={nomorepostaction,decorate}}}

\begin{document}

\title[Metrics on graphs and other structures]{%
  Hausdorff and Wasserstein metrics on graphs and other structured data}
\author[E. Patterson]{Evan Patterson}
\address{Stanford University, Statistics Department}
\email{epatters@stanford.edu}

\begin{abstract}
  Optimal transport is widely used in pure and applied mathematics to find
  probabilistic solutions to hard combinatorial matching problems. We extend the
  Wasserstein metric and other elements of optimal transport from the matching
  of sets to the matching of graphs and other structured data. This
  structure-preserving form of optimal transport relaxes the usual notion of
  homomorphism between structures. It applies to graphs, directed and
  undirected, labeled and unlabeled, and to any other structure that can be
  realized as a $\cat{C}$-set for some finitely presented category $\cat{C}$. We
  construct both Hausdorff-style and Wasserstein-style metrics on $\cat{C}$-sets
  and we show that the latter are convex relaxations of the former. Like the
  classical Wasserstein metric, the Wasserstein metric on $\cat{C}$-sets is the
  value of a linear program and is therefore efficiently computable.
\end{abstract}

\maketitle

\section{Introduction}
\label{sec:introduction}

How do you measure the distance between two graphs or, in a broader sense,
quantify the similarity or dissimilarity of two graphs? Metrics and other
dissimilarity measures are useful tools for mathematicians who study graphs, and
also for practitioners in statistics, machine learning, and other fields who
analyze graph-structured data. Many distances and dissimilarities have been
proposed as part of the general study of graph matching
\parencite{emmert-streib2016,riesen2010}, yet current methods tend to suffer
from one of two problems. Methods that fully exploit the graph structure, such
as graph edit distances and related distances based on maximum common subgraphs
\parencite{bunke1997}, are generally NP-hard to compute and must be approximated
by heuristic search algorithms. Methods based on efficiently computable graph
substructures, such as random walks \parencite{kashima2003}, shortest paths
\parencite{borgwardt2005}, or graphlets \parencite{shervashidze2009}, are
computationally tractable by design but are only sensitive to the particular
substructure under consideration. Most kernel methods for graphs or other
discrete structures fall into this category
\parencite{haussler1999,vishwanathan2010}. Our aim is to construct a metric on
graphs that fully accounts for the graph structure, but attains computational
tractability in a principled way through convex relaxation.

The fundamental difficulty in graph matching is that the optimal correspondence
of vertices between two graphs is unknown and must be estimated from a
combinatorially large set of possibilities. The theory of optimal transport
\parencite{villani2003,villani2008,peyre2019}, now routinely used to find
probabilistic matchings of metric spaces \parencite{memoli2011}, suggests itself
as a general strategy to circumvent this combinatorial problem. Several authors
have proposed specific methods to match graphs or other structured data using
optimal transport \parencite{alvarez2018,courty2017,nikolentzos2017,vayer2018}.

Simplifying somewhat, current applications of optimal transport to graph
matching draw on two major ideas, the Wasserstein distance between measures
supported on a common metric space and the Gromov-Wasserstein distance between
metric measure spaces \parencite{sturm2006a,memoli2011}. If you have a way of
embedding the vertices of two graphs into a common metric space, say a Euclidean
space, then you can compute the Wasserstein distance between these two subspaces
\parencite{nikolentzos2017}. Alternatively, if you have a way of converting each
vertex set into its own metric space, then you can compute the
Gromov-Wasserstein distance between these disjoint spaces. In the latter case,
the distance between two vertices in a graph is often defined to be the length
of the shortest path between them, although there are other possibilities. The
two approaches, via the Wasserstein and Gromov-Wasserstein distances, can also
be combined \parencite{vayer2018}.

Methods of this style reduce the problem of matching graphs to that of matching
metric spaces, and then apply the usual tools of optimal transport for metric
matching. While this may suffice for some purposes, it is conceptually
unsatisfying for the simple reason that graphs are not identical with metric
spaces. Any information that cannot be encoded in the metric is lost to optimal
transport. Thus, if we take the shortest path distance on vertices, the optimal
coupling of vertices depends on the graph's edges only through the lengths of
the shortest paths.

Here we describe a form of optimal transport between graphs that makes no such
reduction. Probabilistic mappings are established between both vertex and edge
sets, and compatibility between the mappings is enforced according to the
nearest analogue of graph homomorphism. These probabilistic graph homomorphisms
are defined as solutions to linear programs and are therefore efficiently
computable.

Our methodology is not ultimately about graphs, but about how the idea of a
\emph{homomorphism}, or structure-preserving map, can be deformed both
probabilistically and metrically. We set forth a general notion of
structure-preserving optimal transport that applies to a limited but important
class of structures. This class encompasses directed, undirected, and bipartite
graphs; graphs with vertex attributes, edge attributes, or both; simplicial
sets, the higher-dimensional generalization of graphs; other variants of graphs,
such as hypergraphs; and unrelated structures. The ensuing optimization problems
are in all cases linear programs.

Other authors have proposed convex or otherwise tractable relaxations of graph
matching, based on spectral methods \parencite{cour2007}, semidefinite
programming \parencite{schellewald2005}, and doubly stochastic matrices
\parencite{aflalo2015}. Closest to ours is the last method, which relaxes the
vertex permutation of a graph isomorphism into a doubly stochastic matrix.
Unlike ours, this method does not straightforwardly generalize from graph
isomorphism to graph homomorphism or to metrics on graphs, nor to graphs with
vertex labels or to structures other than graphs.

\begin{figure}
  \centering
  \begin{tikzcd}[row sep=2em, column sep=-4em]
    & \begin{tabular}{c}
        $\cat{C}$-set morphisms \\ (\cref{sec:c-sets})
    \end{tabular}
    \ar{dl}[above left]{\text{convex relaxation}}
    \ar{dr}{\text{metric relaxation}} & \\
    \begin{tabular}{c}
      Markov $\cat{C}$-set morphisms \\
      (\cref{sec:markov})
    \end{tabular}
    \ar{dr}[below left]{\text{metric relaxation}} & &
    \begin{tabular}{c}
      Hausdorff metric on $\cat{C}$-sets \\
      (\cref{sec:hausdorff})
    \end{tabular}
    \ar{dl}{\text{convex relaxation}} \\
    & \begin{tabular}{c}
      Wasserstein metric on $\cat{C}$-sets \\
      (\cref{sec:wasserstein})
    \end{tabular} & \\
    & \begin{tabular}{c}
      Wasserstein metric on Markov kernels \\
      (\cref{sec:wasserstein-markov})
    \end{tabular}
    \ar{u}{\text{lifts to}} &
  \end{tikzcd}
  \caption{Relations of major concepts and section dependencies}
  \label{fig:outline}
\end{figure}
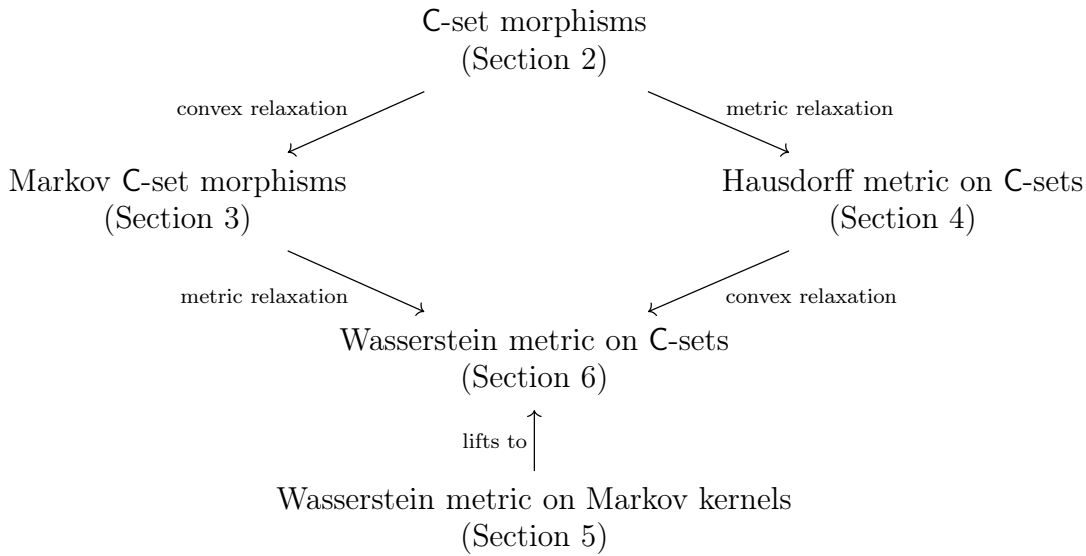

Let us outline in more detail the major concepts and sections of the paper. In
the mainly expository \cref{sec:c-sets}, we review $\cat{C}$-sets, the class of
structures treated throughout. We give numerous examples of $\cat{C}$-sets,
including several kinds of graphs. We also introduce their functorial semantics,
a device we use incessantly to equip $\cat{C}$-sets with probabilistic and
metric structure, beginning in \cref{sec:markov}. There we relax the
$\cat{C}$-set homomorphism problem by replacing functions with their
probabilistic analogue, Markov kernels.\footnotemark{} We arrive at a
feasibility problem reminiscent of optimal transport, although it is expressed
using Markov kernels instead of couplings.

\footnotetext{Readers unfamiliar with Markov kernels will find a brief review in
  \cref{app:markov}.}

We devote the rest of the paper to studying metrics on $\cat{C}$-sets, exact and
probabilistic. \cref{sec:hausdorff} isolates the purely metric aspects of the
problem. We establish a general method for lifting metrics on the hom-sets of a
category $\cat{S}$ to a metric on $\cat{C}$-sets in $\cat{S}$, and we
instantiate the theorem to define a Hausdorff-style metric on $\cat{C}$-sets.
This metric, which generalizes the classical Hausdorff metric on subsets of a
metric space, is generally hard to compute, but may be of theoretical interest.

We then set out to construct a Wasserstein-style metric on $\cat{C}$-sets, using
the same framework. To do this, we must take a detour in
\cref{sec:wasserstein-markov} to define a Wasserstein metric on Markov kernels.
The definition strikes us as very natural, though we can find no source for it
in the literature. It is possibly of independent interest, being the
probabilistic analogue of the $L^p$ metrics and the functional analogue of the
usual Wasserstein metric. Finally, in \cref{sec:wasserstein}, we bring together
these threads to construct a Wasserstein metric on $\cat{C}$-sets. It is a
convex relaxation of the Hausdorff metric on $\cat{C}$-sets and it is, like the
classical optimal transport problem, expressible as a linear program.

The relations between the major concepts are summarized in \cref{fig:outline},
which also serves as a dependency graph for the sections of the paper. If the
meaning of this diagram is not now entirely clear, we hope that it will become
so by the end.

\section{Graphs, \texorpdfstring{$\cat{C}$}{C}-sets, and functorial semantics}
\label{sec:c-sets}

Graphs belong to a class of algebraic structures known as $\cat{C}$-sets. As a
logical system, this class is extremely simple, yet it is broad enough to
encompass a range of useful and important structures, such as directed and
undirected graphs and their higher-dimensional generalizations. It also easily
accommodates the attachment of extra data to arbitrary substructures, as in
vertex- or edge-attributed graphs.

In this section we describe the essential elements of $\cat{C}$-sets, their
morphisms, and their functorial semantics. We give many examples that should be
useful in applications. Most of what we say appears in the literature on
category theory
\parencite{lawvere1989,lawvere2009,maclane1992,reyes2004,spivak2009,spivak2014},
but we assume of the reader nothing more than the definitions of a category, a
functor, and a natural transformation.

\begin{definition}[$\cat{C}$-sets]
  Let $\cat{C}$ be a small category. A \emph{$\cat{C}$-set}\footnotemark{} is a
  functor $X: \cat{C} \to \Set$ from the category $\cat{C}$ to the category of
  sets and functions.

  \footnotetext{$\cat{C}$-sets are more commonly called ``presheaves'' and
    defined to be contravariant functors $\cat{C}^\op \to \Set$. The name
    ``$\cat{C}$-sets'' \parencite{reyes2004} arises from \emph{category
      actions}, which generalize group actions ($G$-sets).}

  Thus, a $\cat{C}$-set $X$ consists of, for every object $c$ in $\cat{C}$, a
  set $X(c)$, and for every morphism $f: c \to c'$ in $\cat{C}$, a function
  $X(f): X(c) \to X(c')$, such that the assignment of functions preserves
  composition and identities.
\end{definition}

Our categories $\cat{C}$ will always have finite presentations, which we regard
as logical theories. A $\cat{C}$-set is then an instance or model of that
theory. A few examples will bring this out.

\begin{example}[Graphs]
  The \emph{theory of graphs} is the category with two objects and two parallel
  morphisms:
  \begin{equation*}
    \Theory{\Graph}  = \left\{
      \begin{tikzcd}
        E \rar[shift left=1]{\src}
          \rar[shift right=1][below]{\tgt}
          & V
      \end{tikzcd}
    \right\}.
  \end{equation*}
  A functor $X: \Theory{\Graph} \to \Set$ consists of a vertex set $X(V)$ and an
  edge set $X(E)$, together with source and target maps
  $X(\src), X(\tgt): X(E) \to X(V)$ which assign the source and target vertices
  of each edge. Thus, a $\Theory{\Graph}$-set is simply a graph.
\end{example}

\begin{figure}
  \centering
  \begin{minipage}{0.5\textwidth}
    \centering
    \begin{tikzpicture}[x=3em,y=2.5em]
      \node[vertex] (v1) at (0,0) {};
      \node[vertex] (v2) at (1,0) {};
      \node[vertex] (v3) at (2,0) {};
      \path[edge] (v1) to [out=45,in=135] (v2);
      \path[edge] (v1) to [out=-45,in=-135] (v2);
      \path[edge] (v2) to (v3);
      \path[edge] (v3) to [loop above,min distance=3em,out=90,in=0,looseness=10] (v3);
      \node[vertex] (v4) at (-0.707,0) {};
      \node[vertex] (v5) at (0,1) {};
      \node[vertex] (v6) at (1,1) {};
      \path[edge] (v4) to (v5);
      \path[edge] (v5) to (v6);
      \node[vertex] (v7) at (2,1) {};
    \end{tikzpicture}
    \captionof{figure}{A graph}
    \label{fig:graph}
  \end{minipage}%
  \begin{minipage}{0.5\textwidth}
    \centering
    \begin{tikzpicture}[x=3.5em,y=3.5em]
      \node[vertex] (v1) at (0,0) {};
      \node[vertex] (v2) at (1,0) {};
      \node[vertex] (v3) at (1,1) {};
      \node[vertex] (v4) at (2,1) {};
      \path[edge] (v1) to [out=15,in=165] (v2);
      \path[edge,dashed] (v2) to [out=-165,in=-15] (v1);
      \path[edge] (v2) to [out=105,in=255] (v3);
      \path[edge,dashed] (v3) to [out=285,in=75] (v2);
      \path[edge] (v3) to [out=15,in=165] (v4);
      \path[edge,dashed] (v4) to [out=-165,in=-15] (v3);
    \end{tikzpicture}
    \captionof{figure}{A symmetric graph}
    \label{fig:symmetric-graph}
  \end{minipage}
\end{figure}

In this paper, a ``graph'' without qualification is a directed graph, possibly
with multiple edges and self-loops (\cref{fig:graph}).\footnotemark{} Different
kinds of graphs arise as $\cat{C}$-sets for different categories $\cat{C}$.

\footnotetext{Such graphs are also called ``directed pseudographs,'' by
  graph theorists, and ``quivers,'' by representation theorists.}

\begin{example}[Symmetric graphs]
  The \emph{theory of symmetric graphs} extends the theory of graphs with an
  edge involution:
  \begin{equation*}
    \setlength{\jot}{0pt}
    \Theory{\SGraph} = \left\{
      \begin{tikzcd}
        E \lar[loop left]{\inv}
          \rar[shift left=1]{\src}
          \rar[shift right=1][below]{\tgt}
          & V
      \end{tikzcd}
      \middle|\
      {\footnotesize
       \begin{aligned}
        \inv^2 = 1_E \\
        \inv \cdot \src = \tgt \\
        \inv \cdot \tgt = \src
      \end{aligned}}
    \right\}.
  \end{equation*}
  A $\Theory{\SGraph}$-set, or \emph{symmetric graph}, is a graph $X$ endowed
  with an involution on edges, that is, a self-inverse, orientation-reversing
  edge map $X(\inv): X(E) \to X(E)$. Loosely speaking, a symmetric graph is a
  graph in which every edge has a matching edge going in the opposite direction
  (\cref{fig:symmetric-graph}). Symmetric graphs are essentially the same as
  undirected graphs.\footnote{In the absence of self-loops, symmetric graphs
    correspond one-to-one with undirected graphs, but a self-loop in an
    undirected graph has two possible representations in a symmetric graph: it
    can be fixed or not under the involution.}
\end{example}

Before giving further examples, we define the notion of homomorphism appropriate
for $\cat{C}$-sets.

\begin{definition}[$\cat{C}$-set morphisms]
  Let $\cat{C}$ be a small category and let $X$ and $Y$ be $\cat{C}$-sets. A
  \emph{morphism of $\cat{C}$-sets} from $X$ to $Y$ is a natural transformation
  $\phi: X \to Y$.

  Thus, a $\cat{C}$-set morphism $\phi: X \to Y$ assigns a function
  $\phi_c: X(c) \to Y(c)$ to each object $c \in \cat{C}$ in such a way that for
  every morphism $f: c \to c'$ in $\cat{C}$, there is a commutative diagram
  \begin{equation*}
    \begin{tikzcd}
      X(c) \rar{X(f)} \dar[][left]{\phi_c}
        & X(c') \dar{\phi_{c'}} \\
      Y(c) \rar{Y(f)}
        & Y(c').
    \end{tikzcd}
  \end{equation*}
\end{definition}

A morphism of graphs, according to this definition, is a graph homomorphism as
ordinarily understood, consisting of a vertex map $\phi_V: X(V) \to Y(V)$ and an
edge map $\phi_E: X(E) \to Y(E)$ that preserves the assignment of source and
target vertices:
\begin{equation*}
  \begin{tikzcd}
    X(E) \rar{\src} \dar[][left]{\phi_E}
      & X(V) \dar{\phi_V} \\
    Y(E) \rar{\src}
      & Y(V)
  \end{tikzcd}
  \hspace{4em}
  \begin{tikzcd}
    X(E) \rar{\tgt} \dar[][left]{\phi_E}
      & X(V) \dar{\phi_V} \\
    Y(E) \rar{\tgt}
      & Y(V).
   \end{tikzcd}
\end{equation*}
In the commutative diagrams, we adopt the convention of writing simply $f$ for
$X(f)$ or $Y(f)$ where no confusion can arise. Similarly, a morphism of
symmetric graphs is a graph homomorphism that preserves the edge involution.

The next example shows that different categories $\cat{C}$ can define
essentially the same $\cat{C}$-sets while yielding genuinely different
$\cat{C}$-set morphisms.

\begin{example}[Reflexive graphs]
  The \emph{theory of reflexive graphs} is
  \begin{equation*}
    \setlength{\jot}{0pt}
    \Theory{\RGraph} = \left\{
      \begin{tikzcd}
        E \rar[shift left=3]{\src}
          \rar[shift right=1][above]{\tgt}
          & V
          \lar[shift left=3][below]{\refl}
      \end{tikzcd}
      \middle|\
      {\footnotesize
       \begin{aligned}
        \refl \cdot \src = 1_V \\
        \refl \cdot \tgt = 1_V
      \end{aligned}}
    \right\}.
  \end{equation*}
  A \emph{reflexive graph} is a graph whose every vertex is endowed with a
  distinguished loop (\cref{fig:reflexive-graph}). As objects, reflexive graphs
  are the same as graphs, inasmuch as they are in one-to-one correspondence with
  each other. However, morphisms of reflexive graphs can ``collapse'' edges into
  vertices by mapping them onto distinguished loops, a possibility not permitted
  of a graph homomorphism. For this reason reflexive graph morphisms are
  sometimes called ``degenerate maps.''
\end{example}

\begin{figure}
  \centering
  \begin{minipage}{0.5\textwidth}
    \centering
    \begin{tikzpicture}[x=3em,y=3em,
        every loop/.style={min distance=3em,out=135,in=45,looseness=10}]
      \node[vertex] (v1) at (0,0) {};
      \node[vertex] (v2) at (1,0) {};
      \node[vertex] (v3) at (1.707,0.707) {};
      \node[vertex] (v4) at (2.414,0) {};
      \path[edge,dashed] (v1) to [loop above] (v1);
      \path[edge,dashed] (v2) to [loop above] (v2);
      \path[edge,dashed] (v3) to [loop above] (v3);
      \path[edge,dashed] (v4) to [loop above] (v4);
      \path[edge] (v1) to (v2);
      \path[edge] (v2) to (v3);
      \path[edge] (v3) to (v4);
    \end{tikzpicture}
    \captionof{figure}{A reflexive graph}
    \label{fig:reflexive-graph}
  \end{minipage}%
  \begin{minipage}{0.5\textwidth}
    \centering
    \begin{tikzpicture}[x=4em,y=1.5em]
      \node[vertex] (u1) at (0,4) {};
      \node[vertex] (u2) at (0,3) {};
      \node[vertex] (u3) at (0,2) {};
      \node[vertex] (u4) at (0,1) {};
      \node[vertex] (v1) at (1,3.5) {};
      \node[vertex] (v2) at (1,2.5) {};
      \node[vertex] (v3) at (1,1.5) {};
      \path[edge] (u1) to (v1);
      \path[edge] (u2) to (v1);
      \path[edge] (u2) to (v2);
      \path[edge] (u4) to [out=20,in=220] (v2);
      \path[edge] (u4) to [out=50,in=200] (v2);
    \end{tikzpicture}
    \captionof{figure}{A bipartite graph}
    \label{fig:bipartite-graph}
  \end{minipage}
\end{figure}

Symmetry and reflexivity combine straightforwardly in \emph{symmetric reflexive
  graphs}, in which the distinguished loops are fixed by the edge involution.
Bipartite graphs form another important class of graphs.

\begin{example}[Bipartite graphs]
  The \emph{theory of bipartite graphs} is
  \begin{equation*}
    \Theory{\BGraph} = \left\{
      \begin{tikzcd}
        U & E \lar[][above]{\src} \rar{\tgt} & V
      \end{tikzcd}
    \right\}.
  \end{equation*}
  A \emph{bipartite graph} $X$ consists of two vertex sets, $X(U)$ and $X(V)$,
  and a set $X(E)$ of edges with sources in $X(U)$ and targets in $X(V)$
  (\cref{fig:bipartite-graph}). A morphism of bipartite graphs $\phi: X \to Y$
  has two vertex maps, $\phi_U: X(U) \to Y(U)$ and $\phi_V: X(V) \to Y(V)$, and
  an edge map $\phi_E: X(E) \to Y(E)$ that preserves the source and target
  vertices.
\end{example}

We have not exhausted the list of graph-like structures that can be defined as
$\cat{C}$-sets. For example, hypergraphs, which generalize graphs by allowing
edges with multiple sources and multiple targets, are $\cat{C}$-sets
\parencite{gallo1993,spivak2009}. So are simplicial sets, the higher-dimensional
analogue of graphs and combinatorial analogue of simplicial complexes.

\begin{example}[Semi-simplicial sets]
  The \emph{semi-simplicial category}, truncated to two dimensions, is
  \begin{equation*}
    \setlength{\jot}{0pt}
    \cat{\Delta}_+^2 := \left\{
      \begin{tikzcd}
        T \rar[shift left=3]{e_0}
          \rar[shift right=1]{e_1}
          \rar[shift right=3][below]{e_2}
         & E \rar[shift left=1]{v_0}
             \rar[shift right=1][below]{v_1}
         & V
      \end{tikzcd}
      \middle|\
      {\footnotesize
       \begin{aligned}
         e_1 v_0 = e_0 v_0 \\
         e_2 v_0 = e_0 v_1 \\
         e_2 v_1 = e_1 v_1
       \end{aligned}}
    \right\}.
  \end{equation*}
  A $\cat{\Delta}_+^2$-set, or \emph{two-dimensional semi-simplicial set}, is a
  collection of triangles, edges, and vertices. Each triangle has three edges,
  in a definite order, and each edge has two vertices, also in a definite order,
  in such a way that the induced assignment of vertices to triangles is
  consistent, according to the simplicial identities
  (\cref{fig:semi-simplicial-set}).

  Semi-simplicial sets up to any dimension $n$, or in all dimensions $n$, can be
  defined as $\cat{C}$-sets, as can several other kinds of simplicial sets
  \parencite{friedman2012,grandis2001,spivak2009}. We will not present the
  simplicial categories here, but we summarize the idea that graphs are
  one-dimensional simplicial sets in the table below.
  \begin{center}
    \vspace{\baselineskip}
    \begin{tabular}{ll}
      \toprule
      1-dimensional & $n$-dimensional \\
      \midrule
      graphs & semi-simplicial sets \\
      reflexive graphs & simplicial sets \\
      symmetric graphs & symmetric semi-simplicial sets \\
      symmetric reflexive graphs & symmetric simplicial sets \\
      \bottomrule
    \end{tabular}
    \vspace{\baselineskip}
  \end{center}
\end{example}

\begin{figure}
  \centering
  \begin{tikzcd}[x=6em,y=6em]
    \node[vertex] (u0) at (0,0) {};
    \node[vertex] (u1) at (0,0.707) {};
    
    \coordinate (v0) at (1,0);
    \coordinate (v1) at (1.5,0.707);
    \coordinate (v2) at (2,0);
    \coordinate (v3) at (2.5,0.707);
    \fill[fill=gray!20] (v0) -- (v1) -- (v2) -- cycle;
    \fill[fill=gray!20] (v1) -- (v2) -- (v3) -- cycle;
    \draw (v0) node[vertex] {};
    \draw (v1) node[vertex] {};
    \draw (v2) node[vertex] {};
    \draw (v3) node[vertex] {};
    \path[edge] (u0) to (v0);
    
    \path[edge] (v1) to node[right,font=\footnotesize] {2} (v0);
    \path[edge] (v2) to node[left,font=\footnotesize] {0} (v1);
    \path[edge] (v2) to node[above,font=\footnotesize] {1} (v0);
    
    \path[edge] (v2) to node[right,font=\footnotesize] {2} (v1);
    \path[edge] (v3) to node[left,font=\footnotesize] {0} (v2);
    \path[edge] (v3) to node[below,font=\footnotesize] {1} (v1);
  \end{tikzcd}
  \caption{A two-dimensional semi-simplicial set}
  \label{fig:semi-simplicial-set}
\end{figure}
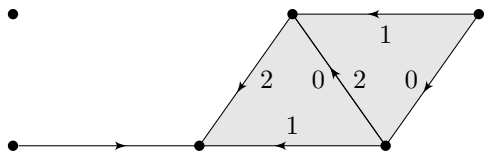

We take this list of examples to establish that many graph-like structures can
be represented as $\cat{C}$-sets. The next example is rather trivial, but later
we use its attributed variant to recover the classical Hausdorff and Wasserstein
distances as special cases of metrics on $\cat{C}$-sets.

\begin{example}[Sets] \label{ex:sets}
  If $\CAT{1} = \{*\}$ is the discrete category on one object, then
  $\CAT{1}$-sets are sets and morphisms of $\CAT{1}$-sets are functions.
\end{example}

\begin{example}[Bisets] \label{ex:bisets} If $\CAT{2}$ is the discrete category
  on two objects, then $\CAT{2}$-sets are pairs of sets and morphisms of
  $\CAT{2}$-sets are pairs of functions.
\end{example}

\begin{example}[Dynamical systems] \label{ex:dds}
  The \emph{theory of discrete dynamical systems} is
  \begin{equation*}
    \Theory{\CAT{DDS}} = \left\{
      \begin{tikzcd}
        * \rar[loop right]{T}
      \end{tikzcd}
    \right\}.
  \end{equation*}
  A \emph{discrete dynamical system} is a set $X = X(*)$ together with a
  function $T: X \to X$. The set $X$ is the state space of the system and the
  transformation $T$ defines the transitions between states.
\end{example}

In applications, graphs and other structures often bear additional data in the
form of discrete labels or continuous measurements. Data attributes are easily
attached to $\cat{C}$-sets by extending the theory $\cat{C}$.

\begin{example}[Attributed sets] \label{ex:attributed-sets}
  The \emph{theory of attributed sets} is
  \begin{equation*}
    \Theory{\CAT{ASet}} = \left\{
      \begin{tikzcd}
        * \rar{\attr} & A
      \end{tikzcd}
    \right\}.
  \end{equation*}
  An \emph{attributed set} is thus a set $X = X(*)$ equipped with an map
  $X \xrightarrow{\attr} X(A)$ that assigns to each element $x \in X$ an
  attribute value $\attr(x)$.
\end{example}

\begin{example}[Vertex-attributed graphs] \label{ex:attributed-graphs}
  The \emph{theory of vertex-attributed graphs} is
  \begin{equation*}
    \Theory{\CAT{VGraph}} = \left\{
      \begin{tikzcd}
        E \rar[shift left=1]{\src}
          \rar[shift right=1][below]{\tgt}
          & V \rar{\attr}
          & A
      \end{tikzcd}
    \right\}.
  \end{equation*}
  A \emph{vertex-attributed graph} is a graph $X$ equipped with a map
  $X(V) \xrightarrow{\attr} X(A)$ that assigns an attribute to each vertex. Such
  graphs are usually called ``vertex-labeled'' when the attribute set $X(A)$ is
  discrete.
\end{example}

Edge-attributed, or vertex- and edge-attributed, graphs can be defined
similarly. Indeed, any number of attributes can be attached to any substructure
of a $\cat{C}$-set, making the class of $\cat{C}$-sets closed under attachment
of extra data.

Often all the $\cat{C}$-sets under consideration take attributes in a common
space $\Attr$, such as a fixed set of labels or a Euclidean space. In this case,
we restrict the $\cat{C}$-set morphisms $\phi: X \to Y$ to those whose attribute
maps $X(A) = \Attr \xrightarrow{\phi_A} \Attr = Y(A)$ are the identity
$1_{\Attr}$. We will note when this restriction is in force by describing the
attribute space as \emph{fixed}.

At this juncture, the reader may wonder what is gained by the formalism of
$\cat{C}$-sets over, say, an equational fragment of first-order logic or even
ordinary, informal mathematics. This question has many valid answers, but the
most pertinent is that viewing theories as categories in their own right makes
it extremely simple to define models with extra structure, be it topological,
metric, measure-theoretic, or otherwise. We simply replace the category $\Set$
of sets and functions with a category $\cat{S}$ having that extra structure.

\begin{definition}[Functorial semantics]
  Let $\cat{C}$ be a small category and let $\cat{S}$ be any category. The
  \emph{functor category} $[\cat{C},\cat{S}]$ has functors $\cat{C} \to \cat{S}$
  as objects and natural transformations between them as morphisms. We call the
  objects of this category \emph{$\cat{S}$-valued $\cat{C}$-sets} or
  \emph{$\cat{C}$-sets in $\cat{S}$}.
\end{definition}

Functorial semantics goes back to Lawvere's pioneering thesis
\parencite{lawvere1963}. When $\cat{S} = \Set$, we recover the original
definitions of $\cat{C}$-sets and $\cat{C}$-set morphisms. So, in our main
example, the category of graphs and graph homomorphisms is the category of
functors $\Theory{\Graph} \to \Set$, that is,
\begin{equation*}
  \Graph \cong [\Theory{\Graph},\Set].
\end{equation*}
The starting point for much subsequent development is the category $\Meas$ of
measurable spaces and measurable functions, defined more carefully below. A
\emph{measurable $\cat{C}$-set}, or $\cat{C}$-set in $\Meas$, is a $\cat{C}$-set
$X$ whose internal sets $X(c)$ are equipped with $\sigma$-algebras and whose
internal maps $X(f)$ are measurable with respect to these $\sigma$-algebras. We
will introduce other categories as we need them. Throughout the paper we explore
the consequences of enriching graphs and other $\cat{C}$-sets with metrics,
measures, and Markov morphisms.

\section{Markov morphisms of measurable \texorpdfstring{$\cat{C}$}{C}-spaces}
\label{sec:markov}

On the topic of matching of $\cat{C}$-sets, the seemingly most elementary
question one can ask is:

\begin{problem}[$\cat{C}$-set homomorphism]
  Given $\cat{C}$-sets $X$ and $Y$, does there exist a $\cat{C}$-set morphism
  $\phi: X \to Y$?
\end{problem}

For $\cat{C} = \CAT{1}$, the problem is trivial. A function $\phi: X \to Y$
exists if and only if the codomain $Y$ is nonempty. But for other categories
$\cat{C}$ the problem is computationally hard. The graph homomorphism problem,
occurring when $\cat{C} = \Theory{\Graph}$, is a famous NP-complete
problem.\footnote{The graph homomorphism problem usually refers to undirected
  graphs, but it is no easier for directed graphs \parencite[Proposition
  5.10]{hell2004}.} In the case of reflexive graphs, the homomorphism problem
becomes trivial again, because there is a reflexive graph morphism $X \to Y$ if
and only if codomain $Y$ is nonempty (contains a vertex). Yet the same cannot be
said of similar matching problems, such as:

\begin{problem}[$\cat{C}$-set isomorphism]
  Given $\cat{C}$-sets $X$ and $Y$, does there exist a $\cat{C}$-set isomorphism
  $X \to Y$?
\end{problem}

The isomorphism problem for reflexive graphs is equivalent to the graph
isomorphism problem, so is once again computationally hard. In summary, while
the complexity depends on the category $\cat{C}$ and on the specific
$\cat{C}$-sets under consideration, it is generally computationally intractable
to find $\cat{C}$-set morphisms, to say nothing of enumerating them or
optimizing over them.

A popular strategy for solving hard combinatorial problems, especially when
inexact solutions are acceptable, is to relax the problem to a continuous one
that is easier to solve. Functorial semantics offer a simple way to implement
this strategy: replace the category $\Set$ with a category having better
computational properties. In what will be a recurring theme, we replace
categories of functions with categories of Markov kernels, which are the
probabilistic analogue of functions. The reader unfamiliar with Markov kernels
will find references and a short review in \cref{app:markov}.

\begin{definition}[Category of Markov kernels]
  The category $\Markov$ has Polish measurable spaces\footnotemark{} as objects
  and Markov kernels between them as morphisms.
\end{definition}

\footnotetext{In other words, the objects are topological spaces homeomorphic to
  a complete, separable metric space, equipped with their Borel
  $\sigma$-algebras. Many results hold under weaker or no assumptions on the
  measurable spaces, but for simplicity we assume this regularity condition
  everywhere.}

\begin{example}[Markov chains] \label{ex:markov-chains}
  A discrete dynamical system in $\Markov$ is a \emph{Markov chain}. Morphisms
  of Markov chains, as stipulated by this definition, are known to probabilists
  as \emph{intertwinings} \parencite{yor1988,diaconis1990}.
\end{example}

Functions are Markov kernels that contain no randomness. To be more formal, a
Markov kernel $M: X \to Y$ is \emph{deterministic} if for every $x \in X$, the
distribution $M(x)$ is a Dirac delta measure. Given any measurable function
$f: X \to Y$, a deterministic Markov kernel $\MarkovFun(f): X \to Y$ is defined
by $\MarkovFun(f)(x) := \delta_{f(x)}$, and every deterministic Markov kernel
arises uniquely in this way. Measurable functions can therefore be identified
with deterministic Markov kernels. Moreover, the identification is functorial.
Given composable measurable functions $X \xrightarrow{f} Y \xrightarrow{g} Z$,
one easily checks that
$\MarkovFun(f \cdot g) = \MarkovFun f \cdot \MarkovFun g$. Also,
$\MarkovFun(1_X) = 1_X$. We summarize these statements by saying that
$\MarkovFun: \Meas \to \Markov$ is an identity-on-objects embedding functor. In
what follows we will not always distinguish notationally between a function $f$
and its corresponding Markov kernel $\MarkovFun(f)$.

Let us now make precise the relaxation of the $\cat{C}$-set homomorphism
problem. Strictly speaking, the relaxation is not from $\cat{C}$-sets, but from
measurable $\cat{C}$-sets.

\begin{definition}[Measurable $\cat{C}$-spaces]
  The category $\Meas$ has Polish measurable spaces as objects and measurable
  functions as morphisms.

  A \emph{measurable $\cat{C}$-space} is a $\cat{C}$-set in $\Meas$.
\end{definition}

The sought-after relaxation is a nearly immediate consequence of the embedding
$\Meas$ in $\Markov$.

\begin{definition}[Markov morphisms]
  A \emph{Markov morphism} $\Phi: X \to Y$ of measurable $\cat{C}$-spaces $X$
  and $Y$ consists of a Markov kernel $\Phi_c: X(c) \to Y(c)$ for each object
  $c \in \cat{C}$, such that for every morphism $f: c \to c'$ in $\cat{C}$, the
  diagram
  \begin{equation*}
    \begin{tikzcd}[column sep=5em]
      X(c) \rar{\MarkovFun(X(f))} \dar[][left]{\Phi_c}
        & X(c') \dar{\Phi_{c'}} \\
      Y(c) \rar{\MarkovFun(Y(f))}
        & Y(c')
    \end{tikzcd}
  \end{equation*}
  in $\Markov$ commutes.
\end{definition}

\begin{proposition}[Relaxation of $\cat{C}$-set homomorphism]
    \label{prop:markov-relaxation}
  Given measurable $\cat{C}$-spaces $X$ and $Y$, the problem of finding a Markov
  morphism $\Phi: X \to Y$ is a convex relaxation of the problem of finding a
  (measurable) morphism $\phi: X \to Y$.
\end{proposition}
\begin{proof}
  The proposition makes two assertions, concerning relaxation and convexity.

  To establish the relaxation, observe that if $\phi: X \to Y$ is a measurable
  morphism, then
  $\MarkovFun(\phi) := (\MarkovFun(\phi_c))_{c \in \cat{C}}: X \to Y$ is a
  Markov morphism, because, by functoriality, the embedding
  $\MarkovFun: \Meas \to \Markov$ preserves naturality squares:
  \begin{equation*}
    \begin{tikzcd}
      X(c) \rar{Xf} \dar[][left]{\phi_c}
        & X(c') \dar{\phi_{c'}} \\
      Y(c) \rar{Yf}
        & Y(c')
    \end{tikzcd}
    \qquad\leadsto\qquad
    \begin{tikzcd}[column sep=5em]
      X(c) \rar{\MarkovFun(Xf)} \dar[][left]{\MarkovFun(\phi_c)}
        & X(c') \dar{\MarkovFun(\phi_{c'})} \\
      Y(c) \rar{\MarkovFun(Yf)}
        & Y(c')
    \end{tikzcd}.
  \end{equation*}
  Thus, if the measurable morphism problem has a solution, so does the Markov
  morphism problem. To state the argument more pithily, the functor
  $\MarkovFun: \Meas \to \Meas$ induces a functor
  $\MarkovFun_*: [\cat{C},\Meas] \to [\cat{C},\Markov]$ by post-composition.

  As for the convexity, the Markov morphism problem,
  \begin{equation*}
    \begin{aligned}
      \text{find}\quad
        & \Phi_c: X(c) \to Y(c),\ c \in \cat{C} \\
      \text{s.t.}\quad
        & Xf \cdot \Phi_{c'} = \Phi_c \cdot Yf,\quad
          \forall f: c \to c' \text{ in } \cat{C},
    \end{aligned}
  \end{equation*}
  is a convex feasibility problem, possibly in infinite dimensions. The
  variables, namely Markov kernels $\Phi_c: X(c) \to Y(c)$ indexed by
  $c \in \cat{C}$, form a convex space, and the constraints are linear in the
  variables.
\end{proof}

One way to think about this result is that the constraints defining a
$\cat{C}$-set morphism, which are only \emph{formally} linear, become
\emph{actually} linear upon relaxation. In the case of greatest practical
interest, when the $\cat{C}$-sets are finite, the result is a linear program.

To make this as transparent as possible, let us write out the linear program.
Given finite $\cat{C}$-sets $X$ and $Y$, identify a function
$X(f): X(c) \to X(c')$ with a binary matrix
$X(f) \in \{0,1\}^{|X(c)| \times |X(c')|}$ whose rows sum to 1 and identify a
Markov kernel $\Phi_c: X(c) \to Y(c)$ with a right stochastic matrix
$\Phi_c \in \R^{|X(c)| \times |Y(c)|}$. The Markov morphism problem is then
\begin{equation*}
  \begin{aligned}
    \text{find}\quad
      & \Phi_c \in \R^{|X(c)| \times |Y(c)|},\ c \in \cat{C} \\
    \text{s.t.}\quad
      & \Phi_c \geq 0,\ \Phi_c \cdot \ones = \ones, \quad
        \forall c \in \cat{C} \\
      & Xf \cdot \Phi_{c'} = \Phi_c \cdot Yf, \quad
        \forall f: c \to c',
  \end{aligned}
\end{equation*}
where $\cdot$ denotes the usual matrix multiplication and $\ones$ denotes the
column vector of all 1's (whose dimensionality is left implicit in the
notation). This feasibility problem is a linear program with linear equality
constraints and nonnegativity constraints.\footnotemark

\footnotetext{The category $\cat{C}$ may contain infinitely many morphisms, but
  it suffices to enforce naturality on a generating set of morphisms. Thus,
  assuming $\cat{C}$ is finitely presented, we can always write the linear
  program with finitely many constraints.}

It will be helpful to see how Markov morphisms behave in a concrete situation.

\begin{figure}
  \captionsetup{width=0.4\textwidth}
  \centering
  \begin{minipage}{0.5\textwidth}
    \centering
    \begin{tikzpicture}
      \node (X) at (0,-0.5) {$X$};
      \node[vertex] (u1) at (0,0) {};
      \node[vertex] (u2) at (0,1) {};
      \node[vertex] (u3) at (0,2) {};
      \draw[edge] (u1) -> (u2);
      \draw[edge] (u2) -> (u3);

      \node (Y) at (3,-0.5) {$Y$};
      \node[vertex] (v1) at (3,0) {};
      \node[vertex] (v2) at (2,1) {};
      \node[vertex] (v3) at (4,1) {};
      \node[vertex] (v4) at (3,2) {};
      \draw[edge] (v1) -> (v2);
      \draw[edge] (v1) -> (v3);
      \draw[edge] (v2) -> (v4);
      \draw[edge] (v3) -> (v4);
    \end{tikzpicture}
    \captionof{figure}{Graphs whose Markov morphisms are mixtures of graph
      homomorphisms}
    \label{fig:graph-markov-mixture}
  \end{minipage}%
  \begin{minipage}{0.5\textwidth}
    \centering
    \begin{tikzpicture}
      \node (X) at (0,-0.5) {$X$};
      \node[vertex] (u) at (0,0) {};
      \path[every loop/.style={min distance=6em,in=150,out=30,looseness=10}]
        (u) edge [edge,loop above] (u);

      \node (Y) at (3,-0.5) {$Y$};
      \node[vertex] (v1) at ($(3,1)+(30:1)$) {};
      \node[vertex] (v2) at ($(3,1)+(150:1)$) {};
      \node[vertex] (v3) at ($(3,1)+(270:1)$) {};
      \path[every edge/.style={edge}]
        (v2) edge[bend left=60] (v1)
        (v2) edge[bend right=60] (v3)
        (v3) edge[bend right=60] (v1);
    \end{tikzpicture}
    \captionof{figure}{Graphs with no homomorphisms or Markov morphisms}
    \label{fig:graph-markov-cycle}
  \end{minipage}
\end{figure}

\begin{example}[Markov morphisms of graphs] \label{ex:markov-graph-morphism}
  Let $X$ and $Y$ be finite graphs. A Markov morphism $\Phi: X \to Y$ is a
  Markov kernel on vertices, $\Phi_V: X(V) \to Y(V)$, and a Markov kernel on
  edges, $\Phi_E: X(E) \to Y(E)$, such that the two diagrams
  \begin{equation*}
    \begin{tikzcd}
      X(E) \rar{\src} \dar[][left]{\Phi_E}
        & X(V) \dar{\Phi_V} \\
      Y(E) \rar{\src}
        & Y(V)
    \end{tikzcd}
    \hspace{4em}
    \begin{tikzcd}
      X(E) \rar{\tgt} \dar[][left]{\Phi_E}
        & X(V) \dar{\Phi_V} \\
      Y(E) \rar{\tgt}
        & Y(V)
    \end{tikzcd}
  \end{equation*}
  in $\Markov$ commute. Since the vertex and edge maps are nondeterministic, it
  does not make sense to ask that the source and target vertices be preserved
  exactly, as in a graph homomorphism. The naturality squares assert the next
  best thing, that for every edge $e$ in $X(E)$, the distribution of $e$'s
  source vertex under $\Phi_V$ is equal to that of the source vertices in the
  edge distribution of $e$ under $\Phi_E$, and similarly for target vertices.

  We bring out the difference between graph homomorphisms and Markov morphisms
  in a series of examples. Between the graphs $X$ and $Y$ of
  \cref{fig:graph-markov-mixture}, there are two graph homomorphisms
  $\phi_1, \phi_2: X \to Y$, corresponding to the two directed paths in $Y$.
  Both are, of course, Markov morphisms, as is any mixture
  $\Phi = t \phi_1 + (1-t) \phi_2$, where $t \in [0,1]$. In this case, every
  Markov morphism is a mixture of graph homomorphisms, so little is lost (or
  gained) by the relaxation. \cref{fig:graph-markov-cycle} presents a similar
  picture. The graph $X$ is a loop and the graph $Y$ is a cycle, though not a
  directed one. There are no Markov graph morphisms from $X$ to $Y$,
  deterministic or otherwise.

  \cref{fig:graph-markov-directed-cycle} looks superficially similar, with the
  graph $Y$ now a directed cycle, but the outcome is more interesting. As
  before, there is no graph homomorphism from $X$ to $Y$, but there is a Markov
  morphism. In fact, if $C_n$ is the directed cycle of length $n$, then for any
  $m,n \geq 1$, a Markov morphism $\Phi: C_m \to C_n$ is given by assigning the
  uniform distributions on all vertices and edges:
  \begin{equation*}
    \Phi_V(v) \sim \Unif(C_n(V)),\ v \in C_m(V), \qquad
    \Phi_E(e) \sim \Unif(C_n(E)),\ e \in C_m(E).
  \end{equation*}
  In particular, it is sometimes possible to find a Markov graph morphism that
  is not a mixture of deterministic morphisms, proving that the notion of Markov
  homomorphism is genuinely weaker than graph homomorphism. That should not be
  surprising, given that the graph homomorphism problem is NP-hard, while the
  Markov graph morphism problem is a linear program, hence solvable in
  polynomial time.

  Finally, \cref{fig:graph-markov-terminal} shows the terminal graph for both
  deterministic and Markov morphisms. Any graph $X$ has a unique graph
  homomorphism, indeed a unique Markov morphism, into the loop $Y$.
\end{example}

\begin{figure}
  \captionsetup{width=0.4\textwidth}
  \centering
  \begin{minipage}{0.5\textwidth}
    \centering
    \begin{tikzpicture}
      \node (X) at (0,-0.5) {$X$};
      \node[vertex] (u) at (0,0) {};
      \path[every loop/.style={min distance=6em,in=150,out=30,looseness=10}]
        (u) edge [edge,loop above] (u);

      \node (Y) at (3,-0.5) {$Y$};
      \node[vertex] (v1) at ($(3,1)+(30:1)$) {};
      \node[vertex] (v2) at ($(3,1)+(150:1)$) {};
      \node[vertex] (v3) at ($(3,1)+(270:1)$) {};
      \path[every edge/.style={edge,bend right=60}]
        (v1) edge (v2)
        (v2) edge (v3)
        (v3) edge (v1);
    \end{tikzpicture}
    \captionof{figure}{Graphs with a Markov morphism, but no graph
      homomorphisms}
    \label{fig:graph-markov-directed-cycle}
  \end{minipage}%
  \begin{minipage}{0.5\textwidth}
    \centering
    \begin{tikzpicture}
      \node (Y) at (0,-0.5) {$Y$};
      \node[vertex] (u) at (0,0) {};
      \path[every loop/.style={min distance=6em,in=150,out=30,looseness=10}]
        (u) edge [edge,loop above] (u);
    \end{tikzpicture}
    \captionof{figure}{The terminal graph, for both homomorphisms and Markov
      morphisms}
    \label{fig:graph-markov-terminal}
  \end{minipage}
\end{figure}

One might suppose that the strategy for relaxing $\cat{C}$-set homomorphism
carries over directly to $\cat{C}$-set isomorphism, but that is not so. The
constraints imposed by isomorphism are bilinear, not linear, so convexity would
be lost in a direct translation. The problem is not just computational, though,
as the following result shows.

\begin{proposition}[{Isomorphism in $\Markov$ \parencite{belavkin2013}}]
  All isomorphisms in $\Markov$ are deterministic. That is, any Markov kernels
  $M: X \to Y$ and $N: Y \to X$ between Polish measurable spaces satisfying
  $M \cdot N = 1_X$ and $N \cdot M = 1_Y$ have the form $M = \MarkovFun(f)$ and
  $N = \MarkovFun(g)$ for measurable functions $f: X \to Y$ and $g: Y \to X$.
\end{proposition}
\begin{proof}
  Under the given assumptions, the extreme points of the convex set
  $\ProbSpace(X)$ of probability measures on $X$ are exactly the point masses
  \parencite[Example 8.16]{simon2011}. If $M: X \to Y$ is an isomorphism in
  $\Markov$, then it acts as a linear isomorphism on $\ProbSpace(X)$ and hence
  preserves the extreme points of $\ProbSpace(X)$. Thus, for every $x \in X$,
  there exists $y \in Y$ such that $M(x) = \delta_x M = \delta_y$, which proves
  that $M$ is deterministic.
\end{proof}

As there is nothing to be gained, computationally or mathematically, by looking
at the isomorphisms in $\Markov$, we will formulate the isomorphism problem in a
different way. Let $X$ and $Y$ be finite $\cat{C}$-sets and equip every set
$X(c)$ and $Y(c)$ with the counting measure. A $\cat{C}$-set morphism
$\phi: X \to Y$ is an isomorphism if and only if each component
$\phi_c: X(c) \to Y(c)$ is an isomorphism of sets (a bijection), and this
happens if and only if each component $\phi_c$ preserves the counting measure,
that is, for every subset $B$ of $Y(c)$, the set $B$ and its preimage
$\phi_c^{-1}(B)$ are of the same size. With this motivation, we define:

\begin{definition}[Measure $\cat{C}$-spaces]
  The category $\Meas_*$ has $\sigma$-finite measures on Polish measurable
  spaces as objects and measurable maps as morphisms. The category $\Markov_*$
  has the same objects and Markov kernels as morphisms.

  A \emph{measure $\cat{C}$-space} is a $\cat{C}$-set in $\Meas_*$.
\end{definition}

Recall that a measurable map $f: (X,\mu_X) \to (Y,\mu_Y)$ of measure spaces is
\emph{measure-preserving} if
\begin{equation*}
  \mu_X f := \mu_x \circ f^{-1} = \mu_Y.
\end{equation*}
Similarly, a Markov kernel $M: (X,\mu_X) \to (Y,\mu_Y)$ between measure spaces
is \emph{measure-preserving} if $\mu_X M = \mu_Y$, as expressed in the operator
notation of \cref{def:markov-operator}. When the measure spaces coincide, that
is, $X = Y$ and $\mu_X = \mu_Y$, the measure $\mu_X$ is also called an
\emph{invariant measure} of $f$ or $M$. Let $\Fix(\Meas_*)$ and
$\Fix(\Markov_*)$ denote the subcategories of $\Meas_*$ and $\Markov_*$ whose
morphisms preserve measure.

\begin{example}[Invariant dynamics] \label{ex:invariant-dynamics}
  A discrete dynamical system in $\Fix(\Meas_*)$ is a \emph{measure-preserving
    dynamical system}, the basic object of study in ergodic theory. A discrete
  dynamical system in $\Fix(\Markov_*)$ is a Markov chain together with an
  invariant measure.
\end{example}

\begin{problem}[Measure-preserving $\cat{C}$-space homomorphism]
  Given measure $\cat{C}$-spaces $X$ and $Y$, does there exist a
  measure-preserving morphism $X \to Y$, i.e., a morphism $\phi: X \to Y$ whose
  components $\phi_c: X(c) \to Y(c)$ are all measure-preserving?
\end{problem}

Due to the motivating case of finite sets and counting measures, the problem of
measure $\cat{C}$-space homomorphism is no easier to solve than $\cat{C}$-set
isomorphism; however, its relaxation to measure-preserving Markov morphism
\emph{is} easier to solve, being a convex feasibility problem.

\begin{proposition}[Relaxation of measure-preserving $\cat{C}$-set morphism]
  Given measure $\cat{C}$-spaces $X$ and $Y$, the problem of finding a
  measure-preserving Markov morphism $\Phi: X \to Y$ is a convex relaxation of
  the problem of finding a measure-preserving morphism $\phi: X \to Y$.
\end{proposition}
\begin{proof}
  For any function $f$ on a measure space $\mu$, we have
  $\mu f = \mu \MarkovFun(f)$. Thus, measure-preserving functions correspond to
  measure-preserving deterministic kernels, and the embedding
  $\MarkovFun: \Meas_* \to \Meas_*$ restricts to an embedding
  $\Fix(\Meas_*) \to \Fix(\Markov_*)$. The relaxation follows by the argument of
  \cref{prop:markov-relaxation}. To prove the convexity, observe that the
  measure-preserving Markov morphism problem,
  \begin{equation*}
    \begin{aligned}
      \text{find}\quad
        & \Phi_c: X(c) \to Y(c),\ c \in \cat{C} \\
      \text{s.t.}\quad
        & \mu_{X(c)} \Phi_{c} = \mu_{Y(c)}, \quad \forall c \in \cat{C} \\
        & Xf \cdot \Phi_{c'} = \Phi_c \cdot Yf,\quad
          \forall f: c \to c' \text{ in } \cat{C},
    \end{aligned}
  \end{equation*}
  merely adds linear equality constraints to the Markov morphism problem.
\end{proof}

As before, when the $\cat{C}$-sets are finite, the feasibility problem is a
linear program.

Insisting that Markov kernels preserve measure brings us closer to classical
optimal transport, formulated in terms of couplings. For any Markov kernel
$M: X \to Y$ preserving finite measures $\mu_X$ and $\mu_Y$, the product measure
$\mu_X \otimes M$ has marginals $\mu_X$ and $\mu_X M = \mu_Y$ and hence is a
coupling of $\mu_X$ and $\mu_Y$. Conversely, if $\pi$ is a product measure on
$X \times Y$ with marginals $\mu_X$ and $\mu_Y$, then by the disintegration
theorem (\cref{thm:disintegration-measure}), there exists a Markov kernel
$M: X \to Y$, unique up to sets of $\mu_X$-measure zero, such that
$\pi = \mu_X \otimes M$, and any such kernel $M$ satisfies $\mu_X M = \mu_Y$.
So, up to null sets, couplings and measure-preserving Markov kernels are the
same. They are even the same as morphisms of measure spaces. Couplings have a
standard composition law, known in optimal transport as the \emph{gluing lemma}
\parencite[Lemma 7.6]{villani2003}, and the composition laws for couplings and
kernels are compatible.

Despite this equivalence, we formulate the content of this paper entirely using
Markov kernels, not couplings. In order to compose couplings, one must first
compute disintegrations, and, while disintegration is a linear operation,
introducing it complicates the optimization problem. Also, and more importantly,
we routinely use Markov kernels that are not measure-preserving. There is no
correspondence between general Markov kernels and couplings. The difference,
roughly speaking, is that Markov kernels are the probabilistic analogue of
functions, while couplings are an analogue of bijections, or of correspondences
\parencite{memoli2011}.

Incidentally, workers in optimal transport have long observed that the
measure-preserving property of a coupling, which in particular requires that the
coupled measures have equal mass, is burdensome in applications. In response
various notions of \emph{unbalanced optimal transport} have been proposed
\parencite[Section 10.2]{peyre2019}. Markov kernels offer another alternative to
couplings.

\section{Hausdorff metric on metric \texorpdfstring{$\cat{C}$}{C}-spaces}
\label{sec:hausdorff}

The $\cat{C}$-set homomorphism problem is too stringent for practical matching
of graphs and other structures. Morphisms of $\cat{C}$-sets, even Markov
morphisms, are all-or-nothing: either they exist or they do not, and when they
do exist, they are distinguished only by coarse qualitative distinctions, like
that of homomorphism versus isomorphism. This is problematic in scientific and
statistical applications, in which the data, be it structural or numerical, is
generally subject to randomness and measurement error. To be tolerant to noise,
we should use an inexact, quantitative measure of structural similarity or
dissimilarity. One approach to dissimilarity, possibly the most important, is
via the ubiquitous mathematical concept of metric.

Throughout the rest of this paper, we develop the metric approach to matching
$\cat{C}$-sets. We will eventually, in \cref{sec:wasserstein}, construct a
computationally tractable, Wasserstein-style metric on $\cat{C}$-sets. In this
section, we focus on the purely metric aspects of the problem. The
Hausdorff-style metric on $\cat{C}$-sets that we propose is generally hard to
compute, but is helpful in isolating the metric concepts from the probabilistic.
It may also be interesting in its own right, as a theoretical tool.

The central idea is to weaken the constraints defining a $\cat{C}$-set
homomorphism from exact equality to approximate equality, with the quality of
approximation determined by a metric on morphisms. Schematically, if $X$ and $Y$
are $\cat{C}$-sets and $\phi: X \to Y$ is a transformation, not necessarily
natural, then for each morphism $f: c \to c'$ in $\cat{C}$, we have a ``lax''
naturality square\footnotemark{}
\begin{equation*}
  \begin{tikzcd}
    X(c) \rar{X(f)} \dar[][left]{\phi_c}
      & X(c') \dar{\phi_{c'}} \dlar[Rightarrow] \\
    Y(c) \rar[][below]{Y(f)}
      & Y(c'),
  \end{tikzcd}
\end{equation*}
where the double arrow represents the value
$d(Xf \cdot \phi_{c'}, \phi_c \cdot Yf)$ of some metric $d$ defined on functions
$X(c) \to Y(c')$. We aggregate these values over all morphisms, or all morphism
generators, $f: c \to c'$ in $\cat{C}$ to obtain an total nonnegative weight for
the transformation. The Hausdorff distance $d_H(X,Y)$ is the weight attained by
optimizing over all transformations $\phi: X \to Y$.

\footnotetext{The mapping $\phi: X \to Y$ can be seen as an enriched lax natural
  transformation. In this section we occasionally use the language of enriched
  category theory \parencite{lawvere1973,kelly1982}, but always in such a way
  that the meaning of the terms is clear from context. We assume no knowledge of
  this subject.}

As a first step in rendering this idea precise, we recall the basic concepts of
metric spaces and metrics on function spaces. It is convenient to work with
a definition of metric that is more general than the classical notion.

\begin{definition}[Metric spaces]
  A \emph{Lawvere metric space}, which we call simply a \emph{metric space}, is
  a set $X$ together with a function $d_X: X \times X \to [0,\infty]$, taking
  values in the extended nonnegative real numbers, that satisfies the identity
  law, $d_X(x,x) = 0$ for all $x \in X$, and the \emph{triangle inequality},
  \begin{equation*}
    d_X(x,x'') \leq d_X(x,x') + d_X(x',x''), \qquad \forall x,x',x'' \in X.
  \end{equation*}
  A metric space $(X,d_X)$ is \emph{classical} if three further axioms are
  satisfied:\footnotemark
  \begin{enumerate}[(i)]
  \item \emph{Finiteness}: $d_X(x,x') < \infty$ for all $x,x' \in X$;
  \item \emph{Positive definiteness}: $x = x'$ whenever $d_X(x,x') = 0$;
  \item \emph{Symmetry}: $d_X(x,x') = d_X(x',x)$ for all $x,x' \in X$.
  \end{enumerate}
  The category $\Met$ has metric spaces as objects and functions between them as
  morphisms.
\end{definition}

\footnotetext{Metrics that fail to satisfy one or more of these axioms occur
  often in metric geometry \parencite{bridson1999,burago2001}, under names like
  ``extended metrics,'' ``pseudometrics,'' and ``quasimetrics.'' The term
  ``Lawvere metric space'' derives from Lawvere's study \parencite{lawvere1973}
  of metric spaces as categories enriched in $[0,\infty]$.}

Unless otherwise noted, all metrics in this paper are in the above generalized
sense.

\begin{example}[Shortest path distance] \label{ex:shortest-path-distance}
  To reprise an example from \cref{sec:introduction}, if $X$ is a graph, then a
  metric is defined on its vertices by letting $d_{X(V)}(v,v')$ be the length of
  the shortest directed path from $v$ to $v'$. The metric is finite if $X$ is
  finite and strongly connected, it is always positive definite, and it is
  symmetric if $X$ is a symmetric graph. Generalizing slightly, if $X$ is a
  weighted graph, where each edge carries a nonnegative weight, a metric on its
  vertices is defined by the shortest weighted path. This metric is positive
  definite if the edge weights are strictly positive.
\end{example}

As outlined above, we will work with $\cat{C}$-sets in categories $\cat{S}$
admitting a measure of distance between morphisms.

\begin{definition}[Metric categories]
  A \emph{metric category} is a category enriched in $\Met$, i.e., a category
  $\cat{S}$ whose hom-sets $\cat{S}(X,Y)$ each have the structure of a metric
  space.
\end{definition}

Under the supremum metric, $\Met$ is itself a metric category. Recall that for
any set $X$ and metric space $Y$, the \emph{supremum metric} on functions
$f,g: X \to Y$ is defined by
\begin{equation*}
  d_\infty(f,g) := \sup_{x \in X} d_Y(f(x),g(x)).
\end{equation*}
When $Y$ is a classical metric space, the supremum metric is also classical when
restricted to bounded functions, i.e., to functions $f: X \to Y$ such that
\begin{equation*}
  \sup_{x \in X} d_Y(y_0,f(x)) < \infty
\end{equation*}
for some (and hence any) $y_0 \in Y$.

Another prominent example of a metric category comes from metric measure spaces.
Later, in \cref{sec:wasserstein}, we will see other examples.

\begin{definition}[Metric measure spaces]
  A \emph{metric measure space}, or \emph{mm space}, is a Polish measurable
  space $X$ together with a metric $d_X$ and a $\sigma$-finite measure $\mu_X$.
  We do not assume that $d_X$ is a classical metric or that it metrizes the
  topology of $X$, although we do require that $d_X$ be lower-semicontinuous
  with respect to the topology of $X$ and so, in particular, be Borel
  measurable.\footnotemark

  The category $\MM$ has metric measure spaces as objects and measurable
  functions as morphisms.
\end{definition}

\footnotetext{Our definition of an mm space is weaker than usual. Most authors
  assume that $d_X$ metrizes the topology of $X$
  \parencite{gromov1999,villani2008,shioya2016}, and some also require that
  $\mu_X$ has full support \parencite{memoli2011}. Here, the Polish topology of
  $X$ serves only as a regularity condition to exclude pathological
  $\sigma$-algebras.}

Under any of the $L^p$ metrics, $\MM$ is a metric category. Recall that for any
measure space $X$ and metric space $Y$, the \emph{$L^p$ metric} on measurable
functions $f,g: X \to Y$ is
\begin{equation*}
  d_{L^p}(f,g) := \begin{cases} \displaystyle
    \left( \int_X d_Y(f(x),g(x))^p\, \mu_X(dx) \right)^{1/p} &
      \text{when $1 \leq p < \infty$} \\ \displaystyle
    \esssup_{x \in X} d_Y(f(x),g(x)) &
      \text{when $p = \infty$}.
  \end{cases}
\end{equation*}
The essential supremum metric differs from the supremum metric only in being
insensitive to sets of $\mu_X$-measure zero.

When $Y$ is a classical metric space, the $L^p$ metrics are also classical when
restricted to the $L^p$ spaces. In this context, the space $L^p(X,Y)$ consists
of equivalence classes of functions $f: X \to Y$ that have finite moments of
order $p$,
\begin{equation*}
  \int_X d_Y(y_0,f(x))^p\, \mu_X(dx) < \infty
  \qquad\text{for some (and hence any) $y_0 \in Y$},
\end{equation*}
when $1 \leq p < \infty$, or that are essentially bounded, when $p = \infty$.
The equivalence relation is that of equality $\mu_X$-almost everywhere.

\begin{definition}[Metric $\cat{C}$-spaces]
  A \emph{metric $\cat{C}$-space} is a $\cat{C}$-set in $\Met$. Likewise, a
  \emph{metric measure $\cat{C}$-space}, or \emph{mm $\cat{C}$-space}, is a
  $\cat{C}$-set in $\MM$.
\end{definition}

As $\Met$ and $\MM$ are metric categories, we will be able to define
quantitative measures of dissimilarity between metric $\cat{C}$-spaces and
between metric measure $\cat{C}$-spaces. However, in order that the
dissimilarity measures be metrics, we must restrict the $\cat{C}$-set
transformations to those whose components do not increase distances. We now
formulate this requirement abstractly, for a general metric category $\cat{S}$.

\begin{definition}[Short morphisms]
  A morphism $f: X \to Y$ in a metric category $\cat{S}$ is \emph{short} if it
  does not increase distances upon pre-composition or post-composition. In other
  words, we have $d(fg,fg') \leq d(g,g')$ for all morphisms $g,g': Y \to Z$ in
  $\cat{S}$, and we have $d(hf,h'f) \leq d(h,h')$ for all morphisms
  $h,h': W \to X$ in $\cat{S}$.
\end{definition}

The class of morphisms in $\cat{S}$ that do not increases distances upon
pre-composition is clearly closed under composition and includes the identities,
and likewise for post-composition. The short morphisms in $\cat{S}$ therefore
form a subcategory of $\cat{S}$, which we denote by $\Short(\cat{S})$.

Characterizing the short morphisms of metric spaces and metric measure spaces is
straightforward. The first example shows that our terminology is consistent with
standard usage.

\begin{proposition}[Short morphisms of metric spaces] \label{prop:short-metric}
  The short morphisms of metric spaces are \emph{short maps},\footnotemark{}
  namely functions $f: X \to Y$ of metric spaces $X,Y$ such that
  \begin{equation*}
    d_Y(f(x),f(x')) \leq d_X(x,x'), \qquad \forall x,x' \in X.
  \end{equation*}
  
  Consequently, $\Short(\Met)$ is the category of metric spaces and short maps.
\end{proposition}
\footnotetext{Other names for short maps are ``contractions,''
  ``distance-decreasing maps,'' ``metric maps,'' and ``nonexpansive maps.''
  Alternatively, short maps are Lipschitz functions with Lipschitz constant 1.
  The category of metric spaces and short maps was first studied by Isbell
  \parencite{isbell1964}, in the case of classical metric spaces, and by Lawvere
  \parencite{lawvere1973}, in the case of generalized metric spaces.}
\begin{proof}
  For \emph{any} functions $f: X \to Y$ and $g,g': Y \to Z$, we have
  \begin{equation*}
    d_\infty(fg,fg')
      = \sup_{y \in f(X)} d_Z(g(y),g'(y))
      \leq \sup_{y \in Y} d_Z(g(y),g'(y))
      = d_\infty(g,g').
  \end{equation*}
  If, moreover, $f$ is a short map, then for any functions $h,h': W \to X$,
  \begin{equation*}
    d_\infty(hf,h'f)
      = \sup_{w \in W} \underbrace{d_Y(f(h(w)), f(h'(w)))}_{\leq d_X(h(w),h'(w))}
      \leq d_\infty(h,h').
  \end{equation*}
  Thus short maps are short morphisms of $\Met$. Conversely, if $f:X \to Y$ is a
  short morphism, let $e_x: I \to X$ denote the unique map from the terminal
  space $I = \{*\}$ onto $x \in X$. Then for any $x,x' \in X$,
  \begin{equation*}
    d_\infty(e_x f, e_{x'} f) \leq d_\infty(e_x,e_{x'})
    \quad\text{if and only if}\quad
    d_Y(f(x),f(x')) \leq d_X(x,x'),
  \end{equation*}
  proving that $f$ is a short map.
\end{proof}

\begin{example}[{Short maps of graphs \parencite[Corollary 1.2]{hell2004}}]
  Let $X$ and $Y$ be graphs with the shortest path distance making the vertex
  sets into metric spaces. For any graph homomorphism $\phi: X \to Y$, the
  vertex map $\phi_V: X(V) \to Y(V)$ is a short map of metric spaces, since
  $\phi$ transforms any path in $X$ into a path in $Y$ of the same length. On
  the other hand, not every short map can be extended to a graph homomorphism
  (consider a constant map). Thus, shorts map of graphs are a weakening of graph
  homomorphisms.
\end{example}

\begin{proposition}[Short morphisms of mm spaces] \label{prop:short-mm}
  The short morphisms between metric measure spaces $X$ and $Y$ are measurable
  maps $f: X \to Y$ that are both measure-decreasing,\footnotemark{}
  \begin{equation*}
    \begin{cases}
      \mu_X f \leq \mu_Y & \quad \text{when $1 \leq p < \infty$} \\
      \mu_X f \ll \mu_Y & \quad \text{when $p = \infty$},
    \end{cases}
  \end{equation*}
  and distance-decreasing,
  \begin{equation*}
    d_Y(f(x),f(x')) \leq d_X(x,x'), \qquad \forall x,x' \in X.
  \end{equation*}
  Consequently, $\Short(\MM)$ is the category of mm spaces and distance- and
  measure-decreasing maps.
\end{proposition}
\footnotetext{A contravariance is involved in describing $\mu_X f \leq \mu_Y$ as
  $f: X \to Y$ being measure-decreasing. In fact, it is the induced set map
  $f^{-1}: \Sigma_Y \to \Sigma_X$ that decreases measure; the map $f$ does not
  act on measurable sets.}
\begin{proof}
  We prove only the case $1 \leq p < \infty$.

  First, we show that $f:X \to Y$ is measure-decreasing if and only if
  pre-composing with $f$ decreases $L^p$ distances. If $f$ is
  measure-decreasing, then for measurable functions $g,g': Y \to Z$,
  \begin{equation*}
    d_{L^p}(fg,fg')^p
      = \int_Y d_Z(g(y),g'(y))^p\, \mu_Xf(dy)
      \leq \int_Y d_Z(g(y),g'(y))^p\, \mu_Y(dy)
      = d_{L^p}(g,g')^p.
  \end{equation*}
  Conversely, if $f$ is not measure-decreasing, then there exists a set
  $B \in \Sigma_Y$ such that $\mu_X f(B) > \mu_Y(B)$. Choose a classical metric
  space $Z$ with at least two points $z$ and $z'$, let $g: Y \to Z$ be the
  constant function at $z$, and let $g': Y \to Z$ be the function equal to $z'$
  on $B$ and equal to $z$ outside of $B$. Then, by construction,
  $d_{L^p}(fg,fg') > d_{L^p}(g,g')$.

  Now we show that $f: X \to Y$ is distance-decreasing (a short map of metric
  spaces) if and only if post-composing with $f$ decreases $L^p$ distances. If
  $f$ is distance-decreasing, then for any measurable functions $h,h': W \to X$,
  \begin{equation*}
    d_{L^p}(hf,h'f)^p
    = \int_W {\underbrace{d_Y(f(h(w)),f(h'(w)))}%
        _{\leq d_X(h(w),h'(w))}}^p\, \mu_W(dw)
    \leq d_{L^p}(h,h')^p.
  \end{equation*}
  For the converse direction, let $I = \{*\}$ be the singleton probability space
  and let $e_x: I \to X$ denote the generalized elements, as in the proof of
  \cref{prop:short-metric}. Then for any $x,x' \in X$,
  \begin{equation*}
    d_{L^p}(e_x f, e_{x'} f) \leq d_{L^p}(e_x,e_{x'})
    \quad\text{if and only if}\quad
    d_Y(f(x),f(x')) \leq d_X(x,x'),
  \end{equation*}
  proving that $f$ is distance-decreasing.
\end{proof}

We saw in \cref{sec:markov} that a measure-preserving function is a surrogate
for a bijection. Likewise, a measure-decreasing function is a surrogate for an
injection, since a function on finite sets is measure-decreasing with respect to
counting measure if and only if it is injective. For functions on finite measure
spaces of equal mass, particularly probability spaces, being measure-decreasing
is the same as being measure-preserving. The elementary version of this fact is
that on finite sets of equal size, injections are the same as bijections.

By definition, a metric category is a category enriched in $\Met$. The
enrichment extends to short morphisms in that for any category $\cat{S}$
enriched in $\Met$, the subcategory $\Short(\cat{S})$ is enriched in
$\Short(\Met)$. We digress slightly to clarify this statement.

\begin{proposition}
  Let $\cat{S}$ be a category. The following are equivalent:
  \begin{enumerate}[(i)]
  \item The category $\cat{S}$ is enriched in $\Short(\Met)$;
  \item For any morphisms
    $\begin{tikzcd}[cramped]
      X \rar[shift left]{f} \rar[shift right][below]{g}
      & Y \rar[shift left]{h} \rar[shift right][below]{k}
      & Z
    \end{tikzcd}$
    in $\cat{S}$,
    \begin{equation*}
      d(fh, gk) \leq d(f,g) + d(h,k);
    \end{equation*}
  \item For any morphisms
    $\begin{tikzcd}[cramped]
      X \rar{f} & Y \rar[shift left]{h} \rar[shift right][below]{k} & Z
    \end{tikzcd}$
    in $\cat{S}$, we have $d(fh,fk) \leq d(h,k)$, and for any morphisms
    $\begin{tikzcd}[cramped]
      X \rar[shift left]{f} \rar[shift right][below]{g} & Y \rar{h} & Z
    \end{tikzcd}$,
    we have  $d(fh,gh) \leq d(f,g)$.
  \end{enumerate}
\end{proposition}
\begin{proof}
  Conditions (i) and (ii) are equivalent by the definition of an enriched
  category. We prove that (ii) and (iii) are equivalent. If (ii) holds, then
  taking $f=g$ gives $d(fh,fk) \leq d(f,f) + d(h,k) = d(h,k)$. Similarly, taking
  $h=k$ gives $d(fh,gh) \leq d(f,g) + d(h,h) = d(f,g)$. Conversely, if (iii)
  holds, then by the triangle inequality,
  \begin{equation*}
    d(fh,gk) \leq d(fh,gh) + d(gh,gk) \leq d(f,g) + d(h,k). \qedhere
  \end{equation*}
\end{proof}

We now define a Hausdorff-style metric on $\cat{C}$-sets in a general metric
category $\cat{S}$. We instantiate it with $\cat{S} = \Met$ and $\cat{S} = \MM$
below and with other metric categories in \cref{sec:wasserstein}. In the
definition, fix $1 \leq p \leq \infty$ and write the $\ell^p$ norm on numbers in
$[0,\infty]$ as the binary operator $\pplus$, meaning that
$a \pplus b = (a^p + b^p)^{1/p}$ when $p < \infty$ and
$a \pplus b = \max\{a,b\}$ when $p = \infty$.

\begin{definition}[Hausdorff metric on $\cat{C}$-sets]
    \label{def:hausdorff-metric}
  Let $\cat{C}$ be a finitely presented category and let $\cat{S}$ be a metric
  category. The \emph{Hausdorff metric} on $\cat{C}$-sets in $\cat{S}$ is given
  by, for $X,Y \in [\cat{C},\cat{S}]$,
  \begin{equation*}
    d_H(X,Y) := \inf_{\phi: X \to Y} \bigpplus_{f: c \to c'}
      d(Xf \cdot \phi_{c'}, \phi_c \cdot Yf),
  \end{equation*}
  where the $\ell^p$ sum is over a fixed, finite generating set of morphisms in
  $\cat{C}$ and the infimum is over all transformations $\phi: X \to Y$, not
  necessarily natural, whose components $\phi_c: X(c) \to Y(c)$ in $\cat{S}$ are
  short morphisms (so belong to $\Short(\cat{S})$).
\end{definition}

The Hausdorff metric is indeed a metric, albeit not a classical one.

\begin{theorem} \label{thm:hausdorff-metric}
  As defined above, the Hausdorff metric on $\cat{C}$-sets in $\cat{S}$ is a
  metric.
\end{theorem}
\begin{proof}
  For any $\cat{C}$-set $X$ in $\cat{S}$, taking the identity transformation
  $1_X: X \to X$ gives
  \begin{equation*}
    d_H(X,X) \leq \bigpplus_{f: c \to c'} d(Xf,Xf) = 0,
  \end{equation*}
  proving that $d_H(X,X) = 0$.

  So we need only verify the triangle inequality. Let $X$, $Y$, and $Z$ be
  $\cat{C}$-sets in $\cat{S}$. Fixing a morphism $f: c \to c'$ in $\cat{C}$,
  define the \emph{weight} of a transformation $\phi: X \to Y$ at $f$ to be the
  number $|\phi|_f := d(Xf \cdot \phi_{c'}, \phi_c \cdot Yf)$. We will show
  that, for any composable transformations
  $X \xrightarrow{\phi} Y \xrightarrow{\psi} Z$ with components in
  $\Short(\cat{S})$, we have a triangle inequality
  \begin{equation*}
    |\phi \cdot \psi|_f \leq |\phi|_f + |\psi|_f.
  \end{equation*}
  The proof then follows readily. By the triangle inequalities for the weights
  $|\blank|_f$ and for the $\ell^p$ sum,
  \begin{equation*}
    d_H(X,Z)
      \leq \bigpplus_{f: c \to c'} |\phi \cdot \psi|_f
      \leq \bigpplus_{f: c \to c'} |\phi|_f +
           \bigpplus_{f: c \to c'} |\psi|_f.
  \end{equation*}
  Take the infimum over the transformations $\phi: X \to Y$ and $\psi: Y \to Z$
  to conclude that $d_H(X,Y) \leq d_H(X,Y) + d_H(Y,Z)$.

  To prove the triangle inequality for the weights, let
  $X \xrightarrow{\phi} Y \xrightarrow{\psi} Z$ be composable transformations
  and consider a ``pasting diagram'' of form
  \begin{equation*}
    \begin{tikzcd}
      X(c) \dar[][left]{\phi_c} & \\
      Y(c) \rar{Y(f)} \dar[][left]{\psi_c}
        & Y(c') \dar{\psi_{c'}} \dlar[Rightarrow] \\
      Z(c) \rar[][below]{Z(f)}
        & Z(c')
    \end{tikzcd}
    \quad + \quad
     \begin{tikzcd}
      X(c) \rar{X(f)} \dar[][left]{\phi_c}
        & X(c') \dar{\phi_{c'}} \dlar[Rightarrow] \\
      Y(c) \rar[][below]{Y(f)}
        & Y(c') \dar{\psi_{c'}}  \\
        & Z(c')
    \end{tikzcd}
    \quad \geq \quad
    \begin{tikzcd}
      X(c) \rar{X(f)} \dar[][left]{\phi_c}
        & X(c') \dar{\phi_{c'}} \dlar[Rightarrow] \\
      Y(c) \rar{Y(f)} \dar[][left]{\psi_c}
        & Y(c') \dar{\psi_{c'}} \dlar[Rightarrow] \\
      Z(c) \rar[][below]{Z(f)}
        & Z(c')
    \end{tikzcd}.
  \end{equation*}
  Here is the argument encoded informally by this diagram, which is easy to
  understand if cumbersome to write down. Since the components of $\phi$ and
  $\psi$ are short morphisms, we have
  \begin{equation*}
    |\phi|_f
      = d(Xf \cdot \phi_{c'}, \phi_c \cdot Yf)
      \geq d(Xf \cdot \phi_{c'} \cdot \psi_{c'}, \phi_c \cdot Yf \cdot \psi_{c'})
  \end{equation*}
  and
  \begin{equation*}
    |\psi|_f
      = d(Yf \cdot \psi_{c'}, \psi_c \cdot Zf)
      \geq d(\phi_c \cdot Yf \cdot \psi_{c'}, \phi_c \cdot \psi_c \cdot Zf).
  \end{equation*}
  Therefore, by the triangle inequality in the hom-space $\cat{S}(X(c),Z(c'))$,
  \begin{align*}
    |\phi|_f + |\psi|_f
      &\geq
        d(Xf \cdot \phi_{c'} \cdot \psi_{c'}, \phi_c \cdot Yf \cdot \psi_{c'}) +
        d(\phi_c \cdot Yf \cdot \psi_{c'}, \phi_c \cdot \psi_c \cdot Zf) \\
      &\geq
        d(Xf \cdot \phi_{c'} \cdot \psi_{c'}, \phi_c \cdot \psi_c \cdot Zf) \\
      &= d(Xf \cdot (\phi \cdot \psi)_{c'}, (\phi \cdot \psi)_c \cdot Zf) \\
      &= |\phi \cdot \psi|_f.
  \end{align*}
  This completes the proof of the triangle inequality for $|\blank|_f$ and hence
  also for $d_H$.
\end{proof}

The assumptions of the theorem can be weakened or strengthened with concomitant
effects on the conclusion.

\begin{remark}[General cost functions]
  As the proof shows, the restriction to short morphisms is what guarantees the
  triangle inequality. Thus, in situations where the triangle inequality is
  inessential, we are free to take the infimum over arbitrary transformations
  $\phi: X \to Y$ with components in $\cat{S}$. We may as well also allow the
  hom-sets of $\cat{S}$ to carry general cost functions, not necessarily
  metrics.
\end{remark}

\begin{remark}[Generators and composites]
  From the coordinate-free perspective of categorical logic, the restriction to
  a finite generating set of morphisms in $\cat{C}$ is open to criticism, since
  it makes the Hausdorff metric depend on how the category $\cat{C}$ is
  presented. Can anything be said about the deviation from naturality for a
  generic morphism in $\cat{C}$? In general, no. If, however, we strengthen the
  assumptions to include $X$ and $Y$ being $\cat{C}$-sets in $\Short(\cat{S})$,
  not merely in $\cat{S}$, then for any composable morphisms
  $c \xrightarrow{f} c' \xrightarrow{g} c''$ in $\cat{C}$, an argument by
  pasting diagram of form
  \begin{equation*}
    \begin{tikzcd}
      X(c) \rar{X(f)} \dar[][left]{\phi_c}
        & X(c') \rar{X(g)} \dar{\phi_{c'}} \dlar[Rightarrow]
        & X(c'') \dar{\phi_{c''}} \dlar[Rightarrow] \\
      Y(c) \rar[][below]{Y(f)}
        & Y(c') \rar[][below]{Y(g)}
        & Y(c'')
    \end{tikzcd}
  \end{equation*}
  shows that for all transformations $\phi: X \to Y$,
  \begin{equation*}
    |\phi|_{f \cdot g} \leq |\phi|_f + |\phi|_g.
  \end{equation*}
  Thus, for $\cat{C}$-sets $X$ and $Y$ in $\Short(\cat{S})$, bounds on the
  weights of $\phi: X \to Y$ at generators yield bounds at composites.
\end{remark}

We emphasize that the Hausdorff metric on $\cat{S}$-valued $\cat{C}$-sets is not
a classical metric. Even when the underlying metrics on $\cat{S}$ are symmetric,
the Hausdorff metric is not, although it can be symmetrized in one of the usual
ways, such as
\begin{equation*}
  \max\{d_H(X,Y),d_H(Y,X)\} \qquad\text{or}\qquad
  \tfrac{1}{2}(d_H(X,Y) + d_H(Y,X)).
\end{equation*}
As a more fundamental matter, the Hausdorff metric is not positive definite,
since $d_H(X,Y) = 0$ whenever there exists a $\cat{C}$-set homomorphism
$\phi: X \to Y$ with components in $\Short(\cat{S})$. Similar statements apply
to the symmetrized metrics. Indeed, under any reasonable definition, the
distance between isomorphic $\cat{C}$-sets should be zero, so we cannot expect
to get positive definiteness without passing to equivalence classes of
isomorphic $\cat{C}$-sets. This is not a matter we will pursue.

We turn now to concrete examples. We frequently need to make a metric space out
of a set with no metric structure. There are several generic ways to do this,
but the most useful is the \emph{discrete metric}, defined on a set $X$ as
\begin{equation*}
  d_X(x,x') := \begin{cases}
    \infty &\text{if $x \neq x'$} \\
    0 &\text{if $x = x'$}.
  \end{cases}
\end{equation*}
A map $f: X \to Y$ out of a discrete metric space $X$ is always
short.\footnotemark{}

\footnotetext{The discrete metric is usually taken to be 1, not $\infty$, off
  the diagonal. Both metrics generate the same topology---the discrete
  topology---but only the infinite discrete metric satisfies a universal
  property in $\Short(\Met)$.}

We can make a $\cat{C}$-set $X$ into a metric $\cat{C}$-space by equipping every
set $X(c)$, $c \in \cat{C}$, with the discrete metric. On such discrete metric
$\cat{C}$-spaces, the Hausdorff metric reduces to the $\cat{C}$-set homomorphism
problem:
\begin{equation*}
  d_H(X,Y) = \begin{cases}
    0 &\text{if there exists a $\cat{C}$-set homomorphism $\phi: X \to Y$} \\
    \infty &\text{otherwise}.
  \end{cases}
\end{equation*}
It can therefore be at least as hard to compute the Hausdorff metric on metric
$\cat{C}$-spaces as it is to solve the $\cat{C}$-set homomorphism problem, which
is generally NP-hard. Overcoming such computational difficulties is a major
motivation for \cref{sec:wasserstein-markov,sec:wasserstein}.

More interesting things happen when the underlying metrics are not all discrete.
As a first example, we show that the classical Hausdorff metric on subsets of a
metric space is a special case of the Hausdorff metric on $\cat{C}$-sets,
justifying our terminology.

\begin{example}[Classical Hausdorff metric] \label{ex:classical-hausdorff}
  Let $\cat{C} = \left\{ * \xrightarrow{\attr} A \right\}$ be the theory of
  attributed sets, as in \cref{ex:attributed-sets}. Let $X$ and $Y$ be
  attributed sets under the discrete metric, with attributes valued in a fixed
  metric space $\Attr = X(A) = Y(A)$. Then $X$ and $Y$ are metric
  $\cat{C}$-spaces and the Hausdorff distance between them is
  \begin{align*}
    d_H(X,Y)
      &= \inf_{f: X \to Y} \sup_{x \in X} d_\Attr(\attr x, \attr f(x)) \\
      &= \sup_{x \in X} \inf_{y \in Y} d_\Attr(\attr x, \attr y),
  \end{align*}
  where the first infimum is over all functions $f: X \to Y$. If
  $X \xrightarrow{\attr} \Attr$ and $Y \xrightarrow{\attr} \Attr$ are injective,
  we can identify $X$ and $Y$ with subsets of the metric space $\Attr$. The
  Hausdorff distance then simplifies to
  \begin{equation*}
    d_H(X,Y) = \sup_{x \in X} \inf_{y \in Y} d_\Attr(x,y),
  \end{equation*}
  which is the classical Hausdorff metric in non-symmetric form. Assuming
  $\Attr$ is a symmetric metric space, we recover the standard Hausdorff metric
  upon symmetrization:
  \begin{equation*}
    \max\{d_H(X,Y), d_H(Y,X)\} = \max\left\{
      \sup_{x \in X} \inf_{y \in Y} d_\Attr(x,y),\,
      \sup_{y \in Y} \inf_{x \in X} d_\Attr(x,y)
    \right\}.
  \end{equation*}
\end{example}

In the next two examples, we define several possible Hausdorff metrics on
graphs.

\begin{example}[Hausdorff metric on attributed graphs]
    \label{ex:hausdorff-attributed-graphs}
  Let $X$ and $Y$ be vertex-attributed graphs, as in
  \cref{ex:attributed-graphs}, with discrete metrics on the vertex and edge sets
  and arbitrary metrics on the attribute sets. Then $X$ and $Y$ are attributed
  graphs in $\Met$ and the Hausdorff distance between them is
  \begin{equation*}
    d_H(X,Y) =
      \inf_{\substack{\phi \in \Graph(X,Y) \\ \phi_A: X(A) \to Y(A)}}
      \sup_{v \in X(V)} d_{Y(A)}(\phi_A(\attr v), \attr \phi_V(v)),
  \end{equation*}
  where the infimum is over graph homomorphisms
  $\phi = (\phi_V,\phi_E): X \to Y$ and short maps $\phi_A: X(A) \to Y(A)$.

  This metric optimizes both the graph homomorphism and the matching of the
  attribute spaces $X(A)$ and $Y(A)$. If instead we fix a metric space $\Attr$
  for the attributes, the Hausdorff distance is
  \begin{equation*}
    d_H(X,Y) = \inf_{\phi \in \Graph(X,Y)} \sup_{v \in X(V)}
      d_\Attr(\attr v, \attr \phi_V(v)),
  \end{equation*}
  where the infimum is now only over graph homomorphisms $\phi: X \to Y$.
\end{example}

\begin{example}[Weak Hausdorff metric on graphs]
    \label{ex:weak-hausdorff-graphs}
  Let $X$ and $Y$ be graphs, with discrete metrics on the edge sets, the
  shortest path distances of \cref{ex:shortest-path-distance} on the vertex
  sets, and counting measures on both vertex and edge sets. Then $X$ and $Y$ are
  graphs in $\MM$ and, for $p = 1$, the Hausdorff distance between them is
  \begin{align*}
    d_H(X,Y) =
      \inf_{\substack{\phi_E: X(E) \to Y(E) \\\phi_V: X(V) \to Y(V)}}
      \sum_{\vertex \in \{\src,\tgt\}}
      \sum_{e \in X(E)} d_{Y(V)}(\phi_V(\vertex e), \vertex \phi_E(e)),
  \end{align*}
  where the infimum is over injective maps $\phi_E: X(E) \to Y(E)$ of the edges
  and injective short maps $\phi_V: X(V) \to Y(V)$ of the vertices. If $X$ is
  monomorphic to $Y$, that is, there is an injective graph homomorphism from $X$
  to $Y$, then $d_H(X,Y) = 0$. But, unlike the previous example, we can still
  have $d_H(X,Y) < \infty$ even when $X$ is not homomorphic to $Y$, because the
  edge map is allowed violate the source and target constraints.

  A concrete example is shown in \cref{fig:approximate-graph-morphism}. The
  2-cycle $X$ is plainly not homomorphic to the 4-cycle $Y$, but, due to the
  pictured transformation, the Hausdorff distance from $X$ to $Y$ is
  $d_H(X,Y) = 2$. In general, if $C_n$ is the directed cycle of length $n$, then
  it can be shown that $d_H(C_m,C_n) = \min\{m,n-m\}$ for any $1 \leq m \leq n$.
  This quantity is the length of the shortest path between the endpoints of an
  $m$-path on the $n$-cycle. In the other direction, $d_H(C_m,C_n) = \infty$
  when $m > n$ since there is no longer any injection from $C_m$ to $C_n$.
\end{example}

\begin{figure}
  \centering
  \begin{tikzpicture}[map/.style={densely dotted,->}]
    \node (X) at (0,-0.5) {$X$};
    \node[vertex] (u1) at ($(0,1)+(0:1)$) {};
    \node[vertex] (u2) at ($(0,1)+(180:1)$) {};
    \coordinate (d1) at ($(0,1)+(45:1)$);
    \coordinate (d2) at ($(0,1)+(315:1)$);
    \path[every edge/.style={edge,bend right=90,looseness=1.50}]
      (u1) edge (u2)
      (u2) edge (u1);

    \node (Y) at (5,-0.5) {$Y$};
    \node[vertex] (v1) at ($(5,1)+(45:1)$) {};
    \node[vertex] (v2) at ($(5,1)+(135:1)$) {};
    \node[vertex] (v3) at ($(5,1)+(225:1)$) {};
    \node[vertex] (v4) at ($(5,1)+(315:1)$) {};
    \coordinate (e1) at ($(5,1)+(105:1)$);
    \coordinate (e2) at ($(5,1)+(205:1)$);
    \path[every edge/.style={edge,bend right=45}]
      (v1) edge (v2)
      (v2) edge (v3)
      (v3) edge (v4)
      (v4) edge (v1);

    \draw[map] (u1) -> (v1);
    \draw[map] (u2) -> (v2);
    \draw[map] (d1) -> (e1);
    \draw[map] (d2) -> (e2);
  \end{tikzpicture}
  \caption{An approximate graph homomorphism between two cycles}
  \label{fig:approximate-graph-morphism}
\end{figure}
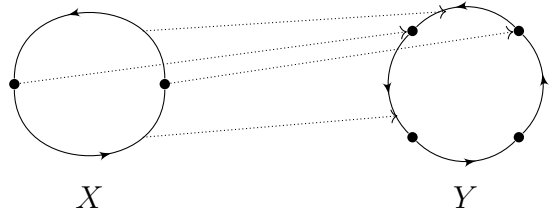

A stream of further examples can be generated by combining the features above,
namely data attributes and metric weakening of the homomorphism constraints, in
graphs or in other $\cat{C}$-sets, such as symmetric graphs, reflexive graphs,
and their higher-dimensional generalizations.

\section{Wasserstein metric on Markov kernels}
\label{sec:wasserstein-markov}

Our goal is now to define a Wasserstein-style metric on metric measure
$\cat{C}$-spaces, thus bringing together the threads of the two preceding
sections. As a first step, we define a metric on Markov kernels, to serve the
same role for the Wasserstein metric as the supremum or $L^p$ metrics do for the
Hausdorff metric. Defining a metric on Markov kernels is more subtle than
defining a metric on functions, and will be the subject of this section.

The Wasserstein metric on Markov kernels generalizes the Wasserstein metric on
probability distributions. Our development parallels that of the classical
metric theory for optimal transport, to be found, for instance, in
\parencite[Chapter 7]{villani2003} or \parencite[Chapter 6]{villani2008}. In
this spirit, we begin with two notions of coupling for Markov
kernels.\footnote{What we call ``products'' of Markov kernels are called
  ``couplings'' in the literature on Markov chains \parencite[Section
  19.1.2]{douc2018}, while our notion of ``coupling'' seems to have no
  established name.}

\begin{definition}[Couplings and products]
  A \emph{coupling} of Markov kernels $M: X \to Y$ and $N: X \to Z$ is any
  Markov kernel $\Pi: X \to Y \times Z$ with marginal $M$ along $Y$ and marginal
  $N$ along $Z$. That is, $\Pi \cdot \proj_Y = M$ and $\Pi \cdot \proj_Z = N$,
  where $\proj_Y: Y \times Z \to Y$ and $\proj_Z: Y \times Z \to Z$ are the
  canonical projection maps.

  A \emph{product} of Markov kernels $M: W \to Y$ and $N: X \to Z$ is any Markov
  kernel $\Pi: W \times X \to Y \times Z$ with marginal $M$ along $W$ and $Y$
  and marginal $N$ along $X$ and $Z$. That is,
  $\Pi \cdot \proj_Y = \proj_W \cdot M$ and
  $\Pi \cdot \proj_Z = \proj_X \cdot N$, where $\proj_W: W \times X \to W$ and
  $\proj_X: W \times X \to X$ are the evident projections.

  The set of all couplings of Markov kernels $M$ and $N$ is denoted by
  $\Coup(M,N)$ and the set of all products by $\Prod(M,N)$.
\end{definition}

To phrase it differently, a Markov kernel $\Pi: X \to Y \times Z$ is a coupling
of $M: X \to Y$ and $N: X \to Z$ if for every $x \in X$, the probability
distribution $\Pi(x)$ is a coupling of $M(x)$ and $N(x)$. Similarly, a Markov
kernel $\Pi: W \times X \to Y \times Z$ is a product of $M: W \to Y$ and
$N: X \to Z$ if for every $w \in W$ and $x \in X$, the distribution $\Pi(w,x)$
is a coupling of $M(w)$ and $N(x)$. In the special case where $W$ and $X$ are
singleton sets, couplings and products of Markov kernels coincide and amount to
couplings of probability measures.

The set of products of Markov kernels $M: W \to Y$ and $N: X \to Z$ is never
empty, because one can always take the \emph{independent product},
\begin{equation*}
  M \otimes N: (w,x) \mapsto M(w) \otimes N(x),
\end{equation*}
given by the pointwise independent product of probability measures. When the
kernels $M$ and $N$ share a common domain $X$, products are extensions of
couplings, because any product $\Pi \in \Prod(M,N)$ gives rise to a coupling
$\Delta_X \cdot \Pi \in \Coup(M,N)$ by pre-composing with the diagonal map
$\Delta_X: x \mapsto (x,x)$. In particular, the set of couplings is never empty
either.

\begin{definition}[Wasserstein metrics on Markov kernels]
    \label{def:wasserstein-markov}
  Let $X$ and $Y$ be metric measure spaces. For any exponent $1 \leq p <
  \infty$, the \emph{Wasserstein metric of order $p$} on Markov kernels $M,N: X
  \to Y$ is
  \begin{equation*}
    W_p(M,N) := \inf_{\Pi \in \Coup(M,N)} \left(
      \int_{X \times Y^2} d_Y(y,y')^p\, \Pi(dy,dy' \given x) \mu_X(dx)
    \right)^{1/p}.
  \end{equation*}
\end{definition}

This metric generalizes two famous constructions in analysis. When the kernels
are deterministic, we recover the $L^p$ metric on functions between metric
measure spaces, reviewed in \cref{sec:hausdorff}. When $X = \{*\}$ is a
singleton set, the kernels $M$ and $N$ can be identified with probability
measures $\mu = M(*)$ and $\nu = N(*)$, and we recover the classical Wasserstein
metric on probability measures,
\begin{equation*}
  W_p(M,N) = W_p(\mu,\nu) = \inf_{\pi \in \Coup(\mu,\nu)} \left(
    \int_{Y^2} d_Y(y,y')^p\, \pi(dy,dy') \right)^{1/p}.
\end{equation*}
The relationships between the base metric and its derived metrics are summarized
in the diagram of \cref{fig:metrics}.

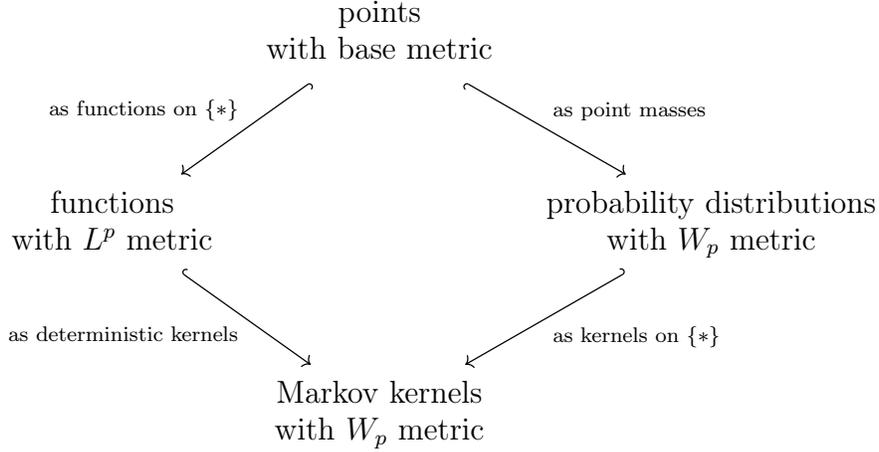
\begin{figure}
  \begin{equation*}
    \begin{tikzcd}[row sep=3em, column sep=0em]
      & \begin{tabular}{c}
          points \\ with base metric
        \end{tabular}
        \ar[hook]{dl}[above left]{\text{as functions on $\{*\}$}}
        \ar[hook']{dr}{\text{as point masses}}
      & \\
      \begin{tabular}{c}
        functions \\ with $L^p$ metric
      \end{tabular}
      \ar[hook]{dr}[below left]{\text{as deterministic kernels}}
      & &
      \begin{tabular}{c}
        probability distributions \\ with $W_p$ metric
      \end{tabular}
      \ar[hook']{dl}{\text{as kernels on $\{*\}$}} \\
      & \begin{tabular}{c}
          Markov kernels \\ with $W_p$ metric
        \end{tabular} &
    \end{tikzcd}
  \end{equation*}
  \caption{Lifting a metric from points to functions, distributions, and Markov
    kernels}
  \label{fig:metrics}
\end{figure}

It is possible to define a Wasserstein metric on Markov kernels in the case
$p=\infty$, generalizing the $L^\infty$ metric on functions and the $W_\infty$
metric on probability measures \parencite{champion2008}, \parencite[Section
3.2]{santambrogio2015}. We do not pursue this case here, as the optimization
problem ceases to be linear in the coupling $\Pi$.

We need to verify that the proposed metric on Markov kernels is actually a
metric. As in the proof for the classical Wasserstein metric, the main property
to verify is the triangle inequality, and crucial step in doing so is
establishing a gluing lemma. Loosely speaking, the \emph{gluing lemma} says that
Markov kernels into $X \times Y$ and $Y \times Z$ that share a common marginal
along $Y$ can be glued along $Y$ to form a Markov kernel into
$X \times Y \times Z$.

\begin{lemma}[Gluing lemma for Markov kernels]
  Let $W$ be a measurable space and $X,Y,Z$ be Polish measurable spaces. Suppose
  $M_X: W \to X$, $M_Y: W \to Y$, and $M_Z: W \to Z$ are Markov kernels, and
  $\Pi_{XY} \in \Coup(M_X,M_Y)$ and $\Pi_{YZ} \in \Coup(M_Y,M_Z)$ are couplings
  thereof. Then there exists a Markov kernel
  $\Pi_{XYZ}: W \to X \times Y \times Z$ with marginals $\Pi_{XY}$ along
  $X \times Y$ and $\Pi_{YZ}$ along $Y \times Z$.
\end{lemma}
\begin{proof}
  By the disintegration theorem for Markov kernels
  (\cref{thm:disintegration-markov}), there exist Markov kernels
  $\Pi_{X \given Y}: W \times Y \to X$ and $\Pi_{Z \given Y}: W \times Y \to Z$
  such that, using a mild abuse of notation,
  \begin{equation*}
    \Pi_{XY}(dx,dy \given w) = \Pi_{X \given Y}(dx \given y,w)\, M_Y(dy \given w)
  \end{equation*}
  and
  \begin{equation*}
    \Pi_{YZ}(dy,dz \given w) = \Pi_{Z \given Y}(dz \given y,w)\, M_Y(dy \given w).
  \end{equation*}
  Then the Markov kernel $\Pi_{XYZ}: W \to X \times Y \times Z$ defined by
  \begin{equation*}
    \Pi_{XYZ}(dx,dy,dz \given w) := \Pi_{X \given Y}(dx \given y,w)\,
      \Pi_{Z \given Y}(dz \given y,w)\, M_Y(dy \given w)
  \end{equation*}
  satisfies the desired properties.
\end{proof}

By a variant of the usual gluing argument, we show that the Wasserstein metric
on Markov kernels is indeed a metric.

\begin{theorem} \label{thm:wasserstein-markov}
  Let $X$ and $Y$ be metric measure spaces. For any $1 \leq p < \infty$, the
  Wasserstein metric of order $p$ on Markov kernels $X \to Y$ is a metric.

  Moreover, if $Y$ is a classical metric space, then the Wasserstein metric is
  also classical when restricted to equivalence classes of Markov kernels with
  finite moments of order $p$, i.e., to Markov kernels $M: X \to Y$ such that
  \begin{equation*}
    \int_{X \times Y} d_Y(y_0, y)^p\, M(dy \given x) \mu_X(dx) < \infty,
  \end{equation*}
  for some (and hence any) $y_0 \in Y$, where we regard Markov kernels $M$ and
  $M'$ as equivalent if $M(x) = M'(x)$ for $\mu_X$-almost every $x \in X$.
\end{theorem}
\begin{proof}
  We prove the triangle inequality first. Let $M_1,M_2,M_3: X \to Y$ be Markov
  kernels and let $\Pi_{12} \in \Coup(M_1,M_2)$ and
  $\Pi_{23} \in \Coup(M_2,M_3)$ be couplings thereof. By the gluing lemma, there
  exists a Markov kernel $\Pi: X \to Y \times Y \times Y$ such that
  $\Pi \cdot \proj_{12} = \Pi_{12}$ and $\Pi \cdot \proj_{23} = \Pi_{23}$, where
  $\proj_{ij}: Y^3 \to Y^2$, $i < j$, are the evident projections. Forming the
  coupling $\Pi_{13} := \Pi \cdot \proj_{13} \in \Coup(M_1,M_3)$, we estimate
  \begin{align*}
    W_p(M_1,M_3)
      &\leq \left(\int d_Y(y_1,y_3)^p\,
        \Pi_{13}(dy_1,dy_3 \given x) \mu_X(dx) \right)^{1/p} \\
      &= \left(\int d_Y(y_1,y_3)^p\,
        \Pi(dy_1,dy_2,dy_3 \given x) \mu_X(dx) \right)^{1/p} \\
      &\leq \left(\int (d_Y(y_1,y_2) + d_Y(y_2,y_3))^p\,
        \Pi(dy_1,dy_2,dy_3 \given x) \mu_X(dx) \right)^{1/p} \\
      &\leq \left(\int d_Y(y_1,y_2)^p\,
        \Pi(dy_1,dy_2,dy_3 \given x) \mu_X(dx) \right)^{1/p} \\
        &\qquad + \left(\int d_Y(y_2,y_3)^p\,
        \Pi(dy_1,dy_2,dy_3 \given x) \mu_X(dx) \right)^{1/p} \\
     &= \left(\int d_Y(y_1,y_2)^p\,
        \Pi_{12}(dy_1,dy_2 \given x) \mu_X(dx) \right)^{1/p} \\
        &\qquad + \left(\int d_Y(y_2,y_3)^p\,
        \Pi_{23}(dy_2,dy_3 \given x) \mu_X(dx) \right)^{1/p},
  \end{align*}
  where we apply the triangle inequality in the second inequality and
  Minkowski's inequality on $L^p(X \times Y^3, \mu_X \otimes \Pi)$ in the third
  inequality. Since the resulting inequality holds for any couplings $\Pi_{12}$
  and $\Pi_{23}$, we conclude that
  \begin{equation*}
    W_p(M_1,M_3) \leq W_p(M_1,M_2) + W_P(M_2,M_3).
  \end{equation*}
  Moreover, for any Markov kernel $M: X \to Y$, the deterministic coupling $M
  \cdot \Delta_Y$, where $\Delta_Y: Y \to Y \times Y$ is the diagonal map,
  yields
  \begin{equation*}
    W_p(M,M)
      \leq \left( \int d_Y(y,y)^p\, M(dy \given x) \mu_X(dx) \right)^{1/p}
      = 0,
    \end{equation*}
  proving that $W_p(M,M) = 0$. Thus $W_p$ is a metric in generalized sense.

  For the second part of the theorem, suppose $Y$ is a classical metric space.
  If Markov kernels $M$ and $M'$ satisfy $W_p(M,M') = 0$, then, assuming for the
  moment that the infimum is achieved, there exists a coupling
  $\Pi \in \Coup(M,M')$ such that
  \begin{equation*}
    \int_{X \times Y \times Y} d_Y(y,y')^p\, \Pi(dy,dy' \given x) \mu_X(dx) = 0.
  \end{equation*}
  Thus, for $\mu_X$-almost every $x \in X$,
  \begin{equation*}
    \int_{Y \times Y} d_Y(y,y')^p\, \Pi(dy,dy' \given x) = 0.
  \end{equation*}
  For each such $x$, since the metric $d_Y$ is positive definite, $\Pi(x)$ is
  concentrated on the diagonal. Thus $\Pi(x) = \nu \cdot \Delta_Y$ for some
  probability measure $\nu$ on $Y$ and hence $\Pi(x) \in \Coup(\nu,\nu)$. But,
  by assumption, $\Pi(x) \in \Coup(M(x),M'(x))$, and so $M(x) = \nu = M'(x)$.
  Thus $M(x) = M'(x)$ for $\mu_X$-a.e. $x \in X$, and we conclude that $W_p$ is
  positive definite. Next, $W_p$ is symmetric since the base metric $d_Y$ is.
  Finally, by taking the independent coupling and using Minkowski's inequality
  again, it is easy to show that $W_p$ is finite on Markov kernels with finite
  moments of order $p$.
\end{proof}

In the second half of the proof, we assumed the following result.

\begin{proposition} \label{prop:wasserstein-existence}
  The infimum in the Wasserstein metric on Markov kernels is attained. That is,
  for any Markov kernels $M,N: X \to Y$ on mm spaces $X,Y$, there exists a
  coupling $\Pi: X \to Y \times Y$ of $M$ and $N$ such that
  \begin{equation*}
    W_p(M,N)^p = \int_{X \times Y \times Y}
      d_Y(y,y')^p\, \Pi(dy,dy' \given x) \mu_X(dx).
  \end{equation*}
\end{proposition}
\begin{proof}
  According to a known existence theorem for optimal couplings
  (\cref{thm:optimal-coupling-markov}), the infimum can even be achieved
  simultaneously at every point $x$. That is, there exists a coupling
  $\Pi \in \Coup(M,N)$ such that for every $x \in X$,
  \begin{equation*}
    W_p(M(x),N(x))^p = \int_{Y \times Y} d_Y(y,y')^p\, \Pi(dy,dy' \given x).
    \qedhere
  \end{equation*}
\end{proof}

The proof even shows that the Wasserstein metric on Markov kernels can be
written in terms of the Wasserstein metric on probability measures as
\begin{equation*}
  W_p(M,N) = \left( \int_X W_p(M(x),N(x))^p\, \mu_X(dx) \right)^{1/p}.
\end{equation*}
Thus, if we view a Markov kernel $X \to Y$ as a function $X \to \ProbSpace(Y)$
of metric spaces, the Wasserstein metric on Markov kernels reduces to the
familiar $L^p$ metric. However, this result depends on the existence theorem for
optimal couplings (\cref{thm:optimal-coupling-markov}), the proof of which is
non-trivial. Wherever possible we prefer to work with the original
\cref{def:wasserstein-markov} in terms of couplings of Markov kernels.

When at least one of the kernels is deterministic, the Wasserstein metric has a
simple, closed-form expression.

\begin{proposition}[Wasserstein metric on deterministic kernels]
    \label{prop:wasserstein-deterministic}
  Let $X,Y,Z$ be metric measure spaces. For any Markov kernel $M: X \to Y$ and
  measurable functions $f: X \to Z$ and $g: Y \to Z$,
  \begin{equation*}
    W_p(f, Mg) = \left( \int_{X \times Y}
      d_Z(f(x),g(y))^p\, M(dy \given x) \mu_X(dx) \right)^{1/p}.
  \end{equation*}
  In particular, for any measurable functions $f,g: X \to Y$,
  \begin{equation*}
    W_p(f,g) = \left( \int_X d_Y(f(x),g(x))^p\, \mu_X(dx) \right)^{1/p}
      = d_{L^p}(f,g).
  \end{equation*}
\end{proposition}
\begin{proof}
  To prove the first statement, notice that $f$ and $M \cdot g$ have a single
  coupling, at once deterministic and independent. By Tonelli's theorem,
  \begin{align*}
    W_p(f, Mg)^p
      &= \int_{X \times Y \times Z \times Z} d_Z(z,z')^p\,
        \delta_{f(x)}(dz) \delta_{g(y)}(dz') M(dy \given x) \mu_X(dx) \\
      &= \int_{X \times Y} d_Z(f(x),g(y))^p \, M(dy \given x) \mu_X(dx).
  \end{align*}
  The second statement follows from the first by taking $X = Y$ and $M = 1_X$.
\end{proof}

\section{Wasserstein metric on metric measure \texorpdfstring{$\cat{C}$}{C}-spaces}
\label{sec:wasserstein}

We are at last ready to construct a Wasserstein-style metric on metric measure
$\cat{C}$-spaces, combining the general metric theory for $\cat{C}$-sets with
the Wasserstein metric on Markov kernels.

Let $\MMarkov$ be the category with metric measures spaces as objects and Markov
kernels as morphisms. In \cref{sec:markov}, we identified measurable functions
with deterministic Markov kernels, obtaining an embedding functor
$\MarkovFun: \Meas \to \Markov$ and thus a relaxation of the $\cat{C}$-set
homomorphism problem. In exactly the same way, the category $\MM$ of metric
measure spaces is functorially embedded inside $\MMarkov$. We denote this
embedding also by $\MarkovFun: \MM \to \MMarkov$. Just as we relaxed the
$\cat{C}$-set homomorphism problem, so will we relax the Hausdorff metric on mm
$\cat{C}$-spaces.

As a first step, we make $\MMarkov$ into a metric category compatible with the
$L^p$ metrics. By \cref{thm:wasserstein-markov}, $\MMarkov$ is a metric category
under the Wasserstein metric of order $p$, for any $1 \leq p < \infty$.
Furthermore, by \cref{prop:wasserstein-deterministic}, this metric agrees with
the corresponding $L^p$ metric on deterministic Markov kernels. Thus, the
embedding $\MarkovFun: \MM \to \MMarkov$ is an isometry of metric categories.

In \cref{sec:hausdorff}, we characterized the short morphisms of $\MM$ under its
$L^p$ metric. The next proposition extends this characterization to $\MMarkov$,
formally reducing to \cref{prop:short-mm} when all the morphisms are
deterministic Markov kernels. Consequently, the isometric embedding functor
$\MarkovFun: \MM \to \MMarkov$ restricts to an embedding
$\Short(\MM) \to \Short(\MMarkov)$ of short morphisms.

\begin{proposition}[Short Markov kernels] \label{prop:short-mmarkov}
  Let $W,X,Y,Z$ be metric measure spaces and let $M: X \to Y$ be a Markov
  kernel.
  \begin{enumerate}[(a)]
  \item $W_p(MN,MN') \leq W_p(N,N')$ for all Markov kernels $N,N': Y \to Z$ if
    and only if $M$ is \emph{measure-decreasing},\footnotemark{}
    \begin{equation*}
      \mu_X M \leq \mu_Y.
    \end{equation*}
  \item $W_p(PM,P'M) \leq W_p(P,P')$ for all Markov kernels $P,P': W \to X$ if
    and only if $M$ is \emph{distance-decreasing} of order $p$, i.e., there
    exists a product $\Pi$ of $M$ with itself such that
    \begin{equation*}
      \left(\int_{Y \times Y} d_Y(y,y')^p\, \Pi(dy,dy' \given x,x')\right)^{1/p}
        \leq d_X(x,x'),
      \qquad
      \forall x,x' \in X.
    \end{equation*}
  \end{enumerate}
  Consequently, a Markov kernel is a short morphism if and only if it is
  distance-decreasing and measure-decreasing. $\Short(\MMarkov)$ is the category
  of mm spaces and distance- and measure-decreasing Markov kernels.
\end{proposition}
\footnotetext{When the measure spaces coincide, that is, $X = Y$ and $\mu_X =
  \mu_Y$, the measure $\mu_X$ is also called a \emph{subinvariant measure} of
  $M$ \parencite[Definition 1.4.1]{douc2018}.}
\begin{proof}
  Towards proving part (a), suppose $M: X \to Y$ is measure-decreasing. For any
  coupling $\Pi: Y \to Z \times Z$ of $N$ and $N'$, the composite
  $M \cdot \Pi: X \to Z \times Z$ is a coupling of $M \cdot N$ and $M\cdot N'$.
  Therefore,
  \begin{align*}
    W_p(MN,MN')^p
      &\leq \int d_Z(z,z')^p\, \Pi(dz,dz' \given y) M(dy \given x) \mu_X(dx) \\
      &= \int d_Z(z,z')^p\, \Pi(dz,dz' \given y) (\mu_X M)(dy) \\
      &\leq \int d_Z(z,z')^p\, \Pi(dz,dz' \given y) \mu_Y(dy).
  \end{align*}
  Since $\Pi \in \Coup(N,N')$ is arbitrary, we have
  $W_p(MN,MN')^p \leq W_p(N,N')^p$.

  For the converse direction, suppose that $M: X \to Y$ is not
  measure-decreasing. Choose a set $Y \in \Sigma_Y$ such that
  $\mu_X M(B) > \mu_Y(B)$. Let $Z$ be a classical metric space with at least two
  points $z$ and $z'$, let $g: Y \to Z$ be the constant function at $z$, and let
  $g': Y \to Z$ be the function equal to $z'$ on $B$ and equal to $z$ outside of
  $B$. The composite $Mg$ is also constant, so by
  \cref{prop:wasserstein-deterministic} we have
  \begin{equation*}
    W_p(Mg,Mg')^p
      = d_Z(z,z')^p \cdot \mu_X M(B)
      > d_Z(z,z')^p \cdot \mu_Y(B)
      = W_p(g,g')^p.
  \end{equation*}

  To prove part (b), suppose $M: X \to Y$ is distance-decreasing, and let $\Pi$
  be a product of $M$ with itself that attains the bound. For any coupling
  $\Lambda: W \to X \times X$ of $P$ and $P'$, the composite
  $\Lambda \cdot \Pi: W \to Y \times Y$ is a coupling of $P \cdot M$ and
  $P' \cdot M$. Therefore,
  \begin{align*}
    W_p(PM,P'M)^p
      &\leq \int_{W \times X} \underbrace{\int_{Y \times Y} d_Y(y,y')^p\,
        \Pi(dy,dy' \given x,x')}_{\leq d_X(x,x')^p}
        \Lambda(dx,dx' \given w) \mu_W(dw) \\
      &\leq \int_{W \times X} d_X(x,x')^p\, \Lambda(dx,x' \given w) \mu_W(dw).
  \end{align*}
  Since $\Lambda \in \Coup(P,P')$ is arbitrary, we obtain
  $W_p(PM,P'M)^p \leq W_p(P,P')^p$.

  Conversely, suppose this inequality holds for all Markov kernels $P,P'$. As in
  the proofs of \cref{prop:short-metric,prop:short-mm}, let $W = I = \{*\}$ be
  the singleton probability space and let $e_x: I \to X$ denote the generalized
  element at $x \in X$. For any $x,x' \in X$, take $P = e_x$ and $P' = e_{x'}$
  to obtain $W_p(M(x'),M(x')) \leq d_X(x,x')$. Using
  \cref{thm:optimal-coupling-markov}, we can construct $\Pi \in \Prod(M,M)$ such
  that
  \begin{equation*}
    \left(\int_{Y \times Y} d_Y(y,y')^p\, \Pi(dy,dy' \given x,x')\right)^{1/p}
      = W_p(M(x),M(x')), \qquad
    \forall x,x' \in X.
  \end{equation*}
  Conclude that $M$ is distance-decreasing.
\end{proof}

The Wasserstein metric on mm $\cat{C}$-spaces is the Hausdorff metric
(\cref{def:hausdorff-metric}) on $\cat{C}$-sets in the metric category
$\MMarkov$. In concrete terms, the definition is:

\begin{definition}[Wasserstein metric on mm $\cat{C}$-spaces]
    \label{def:wasserstein-metric}
  Let $\cat{C}$ be a finitely presented category. For any $1 \leq p < \infty$,
  the \emph{Wasserstein metric of order $p$} on metric measure $\cat{C}$-spaces
  is given by, for $X,Y \in [\cat{C},\MM]$,
  \begin{equation*}
    d_{W,p}(X,Y) := \inf_{\Phi: X \to Y} \left(
      \sum_{f: c \to c'} W_p(Xf \cdot \Phi_{c'}, \Phi_c \cdot Yf)^p
    \right)^{1/p},
  \end{equation*}
  where the sum is over a fixed, finite generating set of morphisms in $\cat{C}$
  and the infimum is over all Markov transformations $\Phi: X \to Y$, not
  necessarily natural, whose components $\Phi_c: X(c) \to Y(c)$ are short
  morphisms (so belong to $\Short(\MMarkov)$).
\end{definition}

In view of the concepts combined, this metric should perhaps be called the
``Hausdorff-Wasserstein metric'' or even, if we are taking the history seriously
\parencite{vershik2013}, the ``Hausdorff-Kantorovich-Rubinstein metric.''
However, we find these names too cumbersome to contemplate.

The proposed metric is indeed a metric, by \cref{thm:hausdorff-metric}. Like any
Hausdorff metric on $\cat{C}$-sets, it is not generally finite, positive
definite, or symmetric, but the failure of positive definiteness is more severe
than before, since $d_{W,p}(X,Y) = 0$ whenever there exists a Markov morphism
$\Phi: X \to Y$ with components in $\Short(\MMarkov)$. More generally, the
Wasserstein metric is a relaxation of the Hausdorff-$L^p$ metric from
\cref{sec:hausdorff}, as clarified by the following inequality.

\begin{proposition}[Wasserstein metric as convex relaxation]
    \label{prop:wasserstein-relaxation}
  For any $1 \leq p < \infty$, the Wasserstein metric of order $p$ on mm
  $\cat{C}$-spaces is a convex relaxation of the Hausdorff-$L^p$ metric on mm
  $\cat{C}$-spaces. That is, for all mm $\cat{C}$-spaces $X$ and $Y$,
  \begin{equation*}
    d_{W,p}(X,Y) \leq d_{H,p}(X,Y),
  \end{equation*}
  where
  \begin{equation*}
    d_{H,p}(X,Y) = \inf_{\phi: X \to Y} \left( \sum_{f: c \to c'}
      d_{L^p}(Xf \cdot \phi_{c'}, \phi_c \cdot Yf)^p
    \right)^{1/p}
  \end{equation*}
  is the Hausdorff metric on $\cat{C}$-sets in the metric category $\MM$ with
  its $L^p$ metric. Moreover, the problem of computing $d_{W,p}(X,Y)^p$ can be
  formulated as a convex optimization problem, possibly in infinite dimensions.
\end{proposition}
\begin{proof}
  We first prove the relaxation. As in the proof of \cref{thm:hausdorff-metric},
  we write $|\phi|_f := d_{L^p}(Xf \cdot \phi_{c'}, \phi_c \cdot Yf)$ for the
  weight of a transformation $\phi: X \to Y$ at a morphism $f: c \to c'$, and,
  similarly, we write $|\Phi|_f := W_p(Xf \cdot \Phi_{c'}, \Phi_c \cdot Yf)$ for
  the weight of a Markov transformation. By the discussion preceding
  \cref{prop:short-mmarkov}, for any transformation $\phi: X \to Y$ with
  components in $\Short(\MM)$, there is a Markov transformation
  $\MarkovFun(\phi) := (\MarkovFun(\phi_c))_{c \in \cat{C}}$ with components in
  $\Short(\MMarkov)$ and of the same weight, $|\MarkovFun(\phi)|_f = |\phi|_f$.
  Consequently,
  \begin{equation*}
    d_{W,p}(X,Y)
      = \inf_{\Phi: X \to Y} \bigpplus_{f: c \to c'} |\Phi|_f
      \leq \inf_{\phi: X \to Y} \bigpplus_{f: c \to c'} |\MarkovFun(\phi)|_f
      = d_{H,p}(X,Y).
  \end{equation*}

  As for the convexity, since the infimum in the Wasserstein metric on Markov
  kernels is attained (\cref{prop:wasserstein-existence}), the value
  $d_{W,p}(X,Y)^p$ is the infimum of
  \begin{equation*}
    \sum_{f: c \to c'} \int d_{Y(c')}(dy,dy')^p\,
      \Pi_f(y,y' \given x) \mu_{X(c)}(dx)
  \end{equation*}
  taken over the Markov kernels
  \begin{equation*}
    \Phi_c: X(c) \to Y(c), \quad
    \Pi_c \in \Prod(\Phi_c,\Phi_c), \quad
    \Pi_f \in \Coup(Xf \cdot \Phi_{c'}, \Phi_c \cdot Yf),
  \end{equation*}
  indexed by objects $c \in \cat{C}$ and morphisms $f: c \to c'$ generating
  $\cat{C}$, and subject to the constraints, for all $c \in \cat{C}$,
  \begin{align*}
    \mu_{X(c)} \cdot \Phi_c &\leq \mu_{Y(c)} \\
    \int d_{Y(c)}(y,y')^p\, \Pi_c(dy,dy' \given x,x') &\leq d_{X(c)}(x,x')^p,
      \quad \forall x,x' \in X(c)
  \end{align*}
  In this optimization problem, the optimization variables belong to convex
  spaces, the objective is linear, and all the constraints, including those for
  couplings and products, are linear equalities or inequalities. The problem is
  therefore convex.
\end{proof}

As in \cref{sec:markov}, the optimization problem is a linear program provided
the $\cat{C}$-sets are finite. Let $X$ and $Y$ be finite metric measure
$\cat{C}$-spaces. Identifying Markov kernels with right stochastic matrices,
measures with row vectors, and exponentiated metrics $d_{X(c)}^p$ with column
vectors $\delta_{X(c)} \in \R^{|Y(c)|^2}$, the distance $d_{W,p}(X,Y)^p$ is the
value of the linear program:
\begin{align*}
  \minimize \quad
    & \sum_{f: c \to c'} \langle \delta_{Y(c)},\, \mu_{X(c)} \Pi_f \rangle
      \displaybreak[0] \\
  \text{over} \quad
    & \Phi_c \in \R^{|X(c)| \times |Y(c)|},\
      \Pi_c \in \R^{|X(c)|^2 \times |Y(c)|^2},\ c \in \cat{C} \\
    & \Pi_f \in \R^{|X(c)| \times |Y(c')|^2},\ f: c \to c' \text{ in } \cat{C}
      \displaybreak[0] \\
  \text{subject to} \quad
    & \Phi_c, \Pi_c, \Pi_f \geq 0,\
      \Phi_c \cdot \ones = \ones,\
      \Pi_c \cdot \ones = \ones,\
      \Pi_f \cdot \ones = \ones, \\
    & \mu_{X(c)} \cdot \Phi_c \leq \mu_{Y(c)}, \\
    & \Pi_c \cdot \delta_{Y{(c)}} \leq \delta_{X(c)}, \\
    & \Pi_c \cdot \proj_1 = \proj_1 \cdot \Phi_c,\
      \Pi_c \cdot \proj_2 = \proj_2 \cdot \Phi_c, \\
    & \Pi_f \cdot \proj_1 = Xf \cdot \Phi_{c'},\
      \Pi_f \cdot \proj_2 = \Phi_c \cdot Yf.
\end{align*}
We leave implicit in the notation the dimensionalities of the projection
operators $\proj_i$ and of column vectors $\ones$ consisting of all 1's.

While it appears forbidding, this linear program simplifies in certain common
situations. When $X(c)$ has the discrete metric, any Markov kernel
$\Phi_c: X(c) \to Y(c)$ is distance-decreasing, so the product $\Pi_c$ and
associated constraints can be eliminated from the program. When $X(c') = Y(c')$
and $\Phi_{c'}: X(c') \to Y(c')$ is fixed to be the identity, as happens for
fixed attribute sets, then for any morphism $f: c \to c'$, the Wasserstein
distance between $Xf$ and $\Phi_c \cdot Yf$ has a closed form
(\cref{prop:wasserstein-deterministic}). The coupling $\Pi_f$ and associated
constraints can thus be eliminated.

Both kinds of simplification occur in the next two examples.

\begin{example}[Classical Wasserstein metric] \label{ex:classical-wasserstein}
  Continuing \cref{ex:attributed-sets,ex:classical-hausdorff}, let $X$ and $Y$
  be attributed sets, equipped with discrete metrics and any probability
  measures, and taking attributes in a fixed mm space $\Attr$. Then $X$ and $Y$
  are attributed sets in $\MM$ and the Wasserstein metric of order $p$ between
  them is
  \begin{align*}
    d_{W,p}(X,Y)
      &= \inf_{M \in \Markov_*(X,Y)} \left(
          \int d_\Attr(\attr x, \attr y)^p\, M(dy \given x) \mu_X(dx)
        \right)^{1/p} \displaybreak[0] \\
      &= \inf_{\pi \in \Coup(\mu_x,\mu_Y)} \left(
          \int d_\Attr(\attr x, \attr y)^p\, \pi(dx,dy)
        \right)^{1/p},
  \end{align*}
  where the second equality follows by the disintegration theorem
  (\cref{thm:disintegration-measure}). In particular, $d_{W,p}(X,Y)^p$ is the
  value of an optimal transport problem. When $X \xrightarrow{\attr} \Attr$ and
  $Y \xrightarrow{\attr} \Attr$ are injective, $X$ and $Y$ can be identified
  with subsets of $\Attr$ and we recover the classical Wasserstein metric,
  namely $d_{W,p}(X,Y) = W_p(\mu_X,\mu_Y)$.
\end{example}

\begin{example}[Wasserstein metric on attributed graphs]
  Continuing \cref{ex:attributed-graphs,ex:hausdorff-attributed-graphs}, let $X$
  and $Y$ be vertex-attributed graphs taking attributes in a fixed mm space
  $\Attr$. Equip the vertex and edge sets with discrete metrics and with any
  fully supported measures. Then $X$ and $Y$ are attributed graphs in $\MM$ and
  the Wasserstein metric of order $p$ between them is
  \begin{equation*}
    d_{W,p}(X,Y) = \inf_{\Phi: X \to Y} \left(
      \int d_\Attr(\attr v, \attr v')^p\, \Phi_V(dv' \given v) \mu_{X(V)}(dv)
    \right)^{1/p},
  \end{equation*}
  where the infimum is over Markov graph morphisms $\Phi: X \to Y$ with
  measure-decreasing components $\Phi_V$ and $\Phi_E$.
\end{example}

This metric takes an infinite value whenever no Markov graph morphism exists.
The next example features a weaker metric, presented, for simplicity, in the
case of unattributed graphs.

\begin{example}[Weak Wasserstein metric on graphs]
  Let $X$ and $Y$ be graphs. Continuing \cref{ex:weak-hausdorff-graphs}, equip
  the edge sets with discrete metrics, the vertex sets with shortest path
  distances, and the vertex and edge sets with counting measures. The
  Wasserstein metric of order 1 is then
  \begin{equation*}
    d_{W,1}(X,Y) =
      \inf_{\substack{\Phi_E: X(E) \to Y(E) \\ \Phi_V: X(V) \to Y(V)}}
      \sum_{\vertex \in \{\src,\tgt\}}
        W_1(X(\vertex) \cdot \Phi_V, \Phi_E \cdot Y(\vertex)),
  \end{equation*}
  where the infimum is over measure-decreasing Markov kernels $\Phi_E$ and
  distance- and measure-decreasing Markov kernels $\Phi_V$. By
  \cref{prop:wasserstein-relaxation}, this metric relaxes the Hausdorff metric
  of \cref{ex:weak-hausdorff-graphs}. It is genuinely weaker: on directed
  cycles, we have $d_{W,1}(C_m,C_n) = 0$ when $1 \leq m \leq n$, as witnessed by
  the uniform Markov graph morphisms of \cref{ex:markov-graph-morphism}. (They
  do not increase distances, like any Markov kernels equal everywhere to a
  constant distribution.) We still have $d_{W,1}(C_m,C_n) = \infty$ when $m > n$
  due to the measure-decreasing constraint.
\end{example}

\section{Conclusion}
\label{sec:conclusion}

We have introduced Hausdorff and Wasserstein metrics on graphs and other
$\cat{C}$-sets and illustrated them through a variety of examples. That being
said, we have established only the most basic properties of the concepts
involved. Possibilities abound for extending this work, in both theoretical and
practical directions. Let us mention a few of them.

Although encompassing graphs, simplicial sets, dynamical systems, and other
structures, the formalism of $\cat{C}$-sets remains possibly the simplest
equational logic, sitting at the bottom of a hierarchy of increasingly
expressive systems \parencite{lambek1986,crole1993}. By admitting categories
$\cat{C}$ with extra structure, such as sums, products, or exponentials, more
complicated structures can be realized as structure-preserving functors from
$\cat{C}$ into $\Set$ or some other category $\cat{S}$. For example, categories
with finite products describe monoids, groups, rings, modules, and other
familiar algebraic structures. This is the original setting of categorical logic
\parencite{lawvere1963}.

Many of the ideas developed here extend to categories $\cat{C}$ with extra
structure. The pertinent questions are whether the extra structure can be
accommodated in the category $\Markov$ and its variants and how this affects the
computation. Sums (coproducts) and units (terminal objects) are easily handled.
$\Markov$ has finite sums and a unit, and since the direct sum of Markov kernels
is a linear operation, it preserves the class of linear optimization problems.
Products are more important and more delicate. $\Markov$ does not have
categorical products, and its natural tensor product, the independent product,
is not a linear operation. In keeping with the spirit of this paper, products in
$\cat{C}$ should be translated into optimal couplings in $\Markov$, resulting in
a larger optimization problem. We leave a proper development of this idea to
future work.

A linear program is solvable in polynomial time and therefore improves
dramatically in tractability over an NP-hard combinatorial problem.
Nevertheless, solving a generic linear program is not always practical. Indeed,
the recent surge in popularity of optimal transport is due partly to the
introduction, by Cuturi and others, of specialized algorithms for solving the
optimal transport program, which far outperform generic interior-point solvers
\parencite{cuturi2013,peyre2019}. It would be useful to know whether and how
these fast algorithms for optimal transport can be adapted to the linear
programs in this paper.

In the new algorithms for optimal transport, the optimization objective is
augmented by a term proportional to the negative entropy of the coupling, a
technique known as \emph{entropic regularization}. With this addition, the
optimal transport problem improves from merely convex to strongly convex and, in
particular, has a unique solution. Besides being useful for optimization,
entropic regularization has a statistical interpretation as Gaussian
deconvolution \parencite{rigollet2018}.

For the Markov morphism feasibility problem of \cref{sec:markov}, adding
entropic regularization yields an optimization problem whose solution is the
Markov morphism of maximum entropy. For instance, in
\cref{fig:graph-markov-mixture}, the maximum entropy Markov morphism is the
uniform mixture $\Phi = \tfrac{1}{2} \phi_1 + \tfrac{1}{2} \phi_2$ of the two
graph homomorphisms $\phi_1$ and $\phi_2$. In
\cref{fig:graph-markov-directed-cycle}, the unique Markov morphism has the
maximum possible entropy, with all vertex and edge distributions being uniform.
As these examples show, entropic regularization is antithetically opposed to the
recovery of deterministic solutions. Entropic regularization should be
investigated more systematically in this context, for algorithmic reasons and
for its intrinsic interest.

\appendix

\section{Markov kernels}
\label{app:markov}

A Markov kernel is the probabilistic analogue of a function, assigning to every
point in its domain not a single point in its codomain but a whole probability
distribution over its codomain. Markov kernels are fundamental objects in
probability theory and statistics
\parencite{cencov1982,kallenberg2002,kallenberg2017,klenke2013}. In this
appendix, we recall the definition and basic properties of Markov kernels, as
well as a few more obscure results from the literature.

\begin{definition}[Markov kernels] \label{def:markov-kernel}
  Let $(X,\Sigma_X)$ and $(Y,\Sigma_Y)$ be measurable spaces. A \emph{Markov
    kernel} from $X$ to $Y$, also known as a \emph{probability kernel} or
  \emph{stochastic kernel}, is a function $M: X \times \Sigma_Y \to [0,1]$ such
  that
  \begin{enumerate}[(i)]
  \item for all points $x \in X$, the map $M(x;\blank): \Sigma_Y \to [0,1]$ is a
    probability measure on $Y$, and
  \item for all sets $B \in \Sigma_Y$, the map $M(\blank; B): X \to [0,1]$ is
    measurable.
  \end{enumerate}
  We usually write $M(dy \given x)$ instead of $M(x;dy)$, in agreement with the
  standard notation for conditional probability.
\end{definition}

Equivalently, a Markov kernel from $X$ to $Y$ is a measurable map
$X\to \ProbSpace(Y)$, where $\ProbSpace(Y)$ is the space of all probability
measures on $Y$ under the $\sigma$-algebra generated by the evaluation maps
$\mu \mapsto \mu(B)$, $B \in \Sigma_Y$ \parencite[Lemma 1.40]{kallenberg2002}.
With this perspective, it is natural to denote a Markov kernel simply as
$M: X \to Y$ and to write $M(x)$ for the distribution $M(x; \blank)$ of $x$
under $M$.

If Markov kernels are probabilistic functions, then they ought to be composable.
They are, according to a standard definition \parencite[Definition
14.25]{klenke2013}.

\begin{definition}[Composition of Markov kernels] \label{def:markov-composition}
  The \emph{composition} of a Markov kernel $M: X \to Y$ with another Markov
  kernel $N: Y \to Z$ is the Markov kernel $M \cdot N: X \to Z$ defined by
  \begin{equation*}
    (M \cdot N)(C \given x) := \int_Y N(C \given y)\, M(dy \given x),
    \qquad x \in X, \quad C \in \Sigma_Z.
  \end{equation*}
  The identity $1_X: X \to X$ with respect to this composition law is the usual
  identity function regarded as a Markov kernel, $1_X: x \mapsto \delta_x$.
\end{definition}

A third perspective on Markov kernels is that they are linear operators on
spaces of measures or probability measures (see, for instance,
\parencite[Theorem 5.2 and Lemma 5.10]{cencov1982} or \parencite[Section
3.3]{worm2010}).

\begin{definition}[Markov operators] \label{def:markov-operator}
  Let $M: X \to Y$ be a Markov kernel. Given a measure $\mu$ on $X$, define the
  measure $\mu M$ on $Y$ by
  \begin{equation*}
    (\mu M)(B) := \int_X M(B \given x)\, \mu(dx), \qquad B \in \Sigma_Y.
  \end{equation*}
  With this definition, $M$ is a \emph{Markov operator}: writing $\MeasSpace(X)$
  for the space of all finite nonnegative measures on $X$, the Markov kernel $M$
  acts as linear map $\MeasSpace(X) \to \MeasSpace(Y)$ that preserves the total
  mass, $\mu M(Y) = \mu(X)$. In particular, $M$ acts as a linear map
  $\ProbSpace(X) \to \ProbSpace(Y)$ on spaces of probability measures.
\end{definition}

Let $\mu$ be a measure on $X$ and let $M: X \to Y$ be a Markov kernel. Besides
applying $M$ to $\mu$, yielding the measure $\mu M$ on $Y$, we can also take the
\emph{product} of $\mu$ and $M$, yielding a measure $\mu \otimes M$ on
$X \times Y$ defined on measurable rectangles by
\begin{equation*}
  (\mu \otimes M)(A \times B) := \int_A M(B \given x)\, \mu(dx),
\end{equation*}
where $A \in \Sigma_X$ and $B \in \Sigma_Y$. The product measure $\mu \otimes M$
has marginal $\mu$ along $X$ and marginal $\mu M$ along $Y$. In the special case
where $M(x)$ equals a fixed measure $\nu$ for all $x$, we recover the usual
product measure $\mu \otimes \nu$.

It is often useful to know that the product operation is invertible. The inverse
operation is called \emph{disintegration}.

\begin{theorem}[{Disintegration of measures
    \parencite[Theorem 1.23]{kallenberg2017}}]
    \label{thm:disintegration-measure}
  Let $X$ be a measurable space and $Y$ be a Polish space. For any
  $\sigma$-finite measure $\pi$ on $X \times Y$, with $\sigma$-finite marginal
  $\mu$ along $X$, there exists a Markov kernel $M: X \to Y$ such that
  $\pi = \mu \otimes M$. Moreover, if $M': X \to Y$ is any other Markov kernel
  with this property, then $M(x) = M'(x)$ for $\mu$-almost every $x \in X$.
\end{theorem}

What is less well known is that Markov kernels into product spaces can also be
disintegrated. We use this result in \cref{sec:wasserstein-markov} to prove a
gluing lemma for Markov kernels.

\begin{theorem}[{Disintegration of Markov kernels
    \parencite[Theorem 1.25]{kallenberg2017}}]
    \label{thm:disintegration-markov} 
  Let $X$ be a measurable space and let $Y$ and $Z$ be Polish spaces. For any
  Markov kernel $\Pi: X \to Y \times Z$ with marginal $M: X \to Y$ along $Z$,
  there exists a Markov kernel $\Lambda: X \times Y \to Z$ such that
  $\Pi = M \otimes \Lambda$, where the product $M \otimes \Lambda$ is defined
  on measurable rectangles by
  \begin{equation*}
    (M \otimes \Lambda)(B \times C \given x) :=
      \int_B \Lambda(C \given x, y)\, M(dy \given x),
  \end{equation*}
  where $x \in X$, $B \in \Sigma_Y$, and $C \in \Sigma_Z$.
\end{theorem}

In the special case where $X = \{*\}$ is a singleton, a Markov kernel
$\Pi: X \to Y \times Z$ is a probability distribution on $Y \times Z$ and we
recover a version of the previous theorem.

Under mild assumptions, optimal couplings of probability measures exist,
according to a standard result in optimal transport \parencite[Theorem
4.1]{villani2008}. Markov kernels have optimal couplings under the same
assumptions, though proving it is more involved. Versions of the following
theorem appear in the literature as \parencite[Theorem 20.1.3]{douc2018},
\parencite[Corollary 5.22]{villani2008}, and \parencite[Theorem 1.1]{zhang2000}.

\begin{theorem}[Optimal coupling of Markov kernels]
    \label{thm:optimal-coupling-markov}
  Let $X$ be a measurable space, $Y$ be a Polish space, and
  $c: Y \times Y \to \R$ be a nonnegative, lower semi-continuous cost function.
  For any Markov kernels $M,N: X \to Y$, there exists a Markov kernel
  $\Pi: X \times X \to Y \times Y$ such that, for every $x, x' \in X$,
  \begin{enumerate}[(i)]
  \item $\Pi(x,x')$ is a coupling of $M(x)$ and $N(x')$, and
  \item $\Pi(x,x')$ is optimal with respect to the cost function $c$, i.e.,
    \begin{equation*}
      \int_{Y \times Y} c(y,y')\, \Pi(dy,dy' \given x,x') =
        \inf_{\pi \in \Coup(M(x),N(x'))} \int_{Y \times Y} c(y,y')\, \pi(dy,dy').
    \end{equation*}
  \end{enumerate}
\end{theorem}

We invoke this theorem in \cref{sec:wasserstein-markov} to show that the infimum
defining the Wasserstein metric on Markov kernels is attained.

\printbibliography

\end{document}